\documentclass[11pt]{amsart}
\usepackage[margin=1.25in,letterpaper,portrait]{geometry}
\usepackage[style=alphabetic]{biblatex}
\bibliography{ref-fac-gen}
\usepackage{latexsym,amssymb,verbatim,multirow,graphicx,epsfig,enumerate,
  paralist,
  hyperref}
\usepackage{graphbox}
\usepackage{subcaption}
\usepackage{xcolor}
\usepackage[normalem]{ulem}
\usepackage{tikz}
\usepackage{bm}
\usepackage{float}
\usepackage{multirow}
\usepackage{breqn}
\usepackage{thmtools}
\usepackage{mathtools}

\usepackage{thm-restate}

\theoremstyle{plain}
   \newtheorem{theorem}{Theorem}[section]
   \newtheorem{proposition}[theorem]{Proposition}
   
   \newtheorem{lemma}[theorem]{Lemma}
   \newtheorem*{corollary*}{Corollary}
   \newtheorem{corollary}[theorem]{Corollary}

   \newtheorem*{theorem*}{Theorem}
\theoremstyle{definition}
   
   \newtheorem{definition}[theorem]{Definition}
   \newtheorem{example}[theorem]{Example}

   \newtheorem{remark}[theorem]{Remark}

\numberwithin{equation}{section}

\usepackage{xparse}

\DeclareDocumentCommand \ltr { o } {%
  \IfNoValueTF {#1} {%
    \ell_W^{\mathrm{full}} %
  }{%
    \ell_{#1}^{\mathrm{full}}%
  }%
}
\DeclareDocumentCommand \lR { o } {%
  \IfNoValueTF {#1} {%
    \ell_W^{\mathrm{red}} %
  }{%
    \ell_{#1}^{\mathrm{red}}%
  }%
}

\DeclareDocumentCommand \Fred { o } {%
  \IfNoValueTF {#1} {%
    F_W^{\mathrm{red}} %
  }{%
    F_{#1}^{\mathrm{red}}%
  }%
}

\newcommand{\op}[1]{\operatorname{#1}}

\newcommand{\Ftr}{F^\mathrm{full}}
\newcommand{\FFFtr}{\FFF^{\mathrm{full}}}

\newcommand\Symm{\mathfrak{S}}

\newcommand\codim{\operatorname{codim}}

\newcommand\wt{\col}
\newcommand{\col}{\operatorname{col}}

\newcommand{\FFF}{\mathcal{F}}
\newcommand{\RRR}{\mathcal{R}}

\newcommand{\defn}[1]{{\color{blue} \it {#1}}}

\newcommand\CC{{\mathbb{C}}}
\newcommand\ZZ{{\mathbb{Z}}}

\newcommand{\id}{\op{id}}

\begin{document}

\title[$W$-Hurwitz numbers I]{Hurwitz numbers for reflection groups I: Generatingfunctionology}
\author{Theo Douvropoulos, Joel Brewster Lewis, Alejandro H. Morales}

\maketitle

\begin{abstract}
The classical Hurwitz numbers count the fixed-length transitive transposition factorizations of a permutation, with a remarkable product formula for the case of minimum length (\emph{genus} $0$). We study the analogue of these numbers for reflection groups with the following generalization of transitivity: say that a reflection factorization of an element in a reflection group $W$ is \emph{full} if the factors generate the whole group $W$. We compute the generating function for full factorizations of arbitrary length for an arbitrary element in a group in the combinatorial family $G(m, p, n)$ of complex reflection groups in terms of the generating functions of the symmetric group $\Symm_n$ and the cyclic group of order $m/p$. As a corollary, we obtain leading-term formulas which count minimum-length full reflection factorizations of an arbitrary element in $G(m,p,n)$ in terms of the Hurwitz numbers of genus $0$ and $1$ and number-theoretic functions. We also study the structural properties of such generating functions for any complex reflection group; in particular, we show via representation-theoretic methods that they can by expressed as finite sums of exponentials of the variable.

\end{abstract}

\section{Introduction}

A factorization $t_1\cdots t_k=\sigma$ of a given element $\sigma$ in the symmetric group $\Symm_n$ as a product of \emph{transpositions} $t_i$ is said to be \defn{transitive} if the group $\langle t_i\rangle_{i=1}^k$ generated by the factors acts transitively on the set $\{1,\dots,n\}$.  Such factorizations arose initially in the work of Hurwitz, in connection with his study of Riemann surfaces.  Hurwitz gave a sketch of an inductive argument, reproduced in detail in \cite{strehl}, for the following remarkable product formula.

\begin{theorem}[{Hurwitz formula \cite{Hurwitz}}]
\label{thm:S_n genus 0}
  The minimum length of a transitive transposition factorization in $\Symm_n$ of a permutation of cycle type $\lambda:=(\lambda_1,\ldots,\lambda_k)$ is $n+k-2$. The number of such minimum-length factorizations is
\[
H_0(\lambda) =  (n+k-2)!\cdot n^{k-3}\cdot  \prod_{i=1}^k \frac{\lambda_i^{\lambda_i}}{(\lambda_i-1)!}.
\]
\end{theorem}

The formula in Theorem~\ref{thm:S_n genus 0}, rediscovered by Goulden and Jackson in \cite{GJ97}, is for what are now called the \defn{(single) Hurwitz numbers of genus $0$}. These numbers also count certain connected planar graphs embedded in the sphere (the {\em planar maps}; see, e.g., \cite{LandoZvonkin}).  In general, the \defn{genus-$g$ Hurwitz numbers} $H_g(\lambda)$ count transitive factorizations $ t_1\cdots  t_c$ of a fixed element $\sigma$ in $\Symm_n$ of cycle type $\lambda$ into $c=n+k+2g-2$ transpositions $ t_i$, as well as certain embedded maps in orientable surfaces of genus $g$.

In the 1980s, work of Stanley \cite{St80} and Jackson \cite{J88} rekindled the interest in the enumeration of factorizations in $\Symm_n$, unrelated to the original topological context. Independently, the next few decades saw the emergence of \emph{Coxeter combinatorics}; one of its main breakthroughs was the realization that theorems about $\Symm_n$ are often shadows of more general results that hold for all reflection groups $W$. In the context of factorizations, this means replacing transpositions in $\Symm_n$ with \emph{reflections} in $W$.

The intersection of these two areas has witnessed a lot of research activity recently \cite{CC,D2,LM,PR}, especially for factorizations of \emph{Coxeter elements} in $W$, generalizing the long cycle case $\lambda=(n)$. In general, however, analogues of Theorem~\ref{thm:S_n genus 0} have been hard to find, not least because it is unclear how to define transitivity in reflection groups (see Section~\ref{sec: transitivity}): in $\Symm_n$, transitivity corresponds to the \emph{connectedness} of the associated embedded map or Riemann surface, but these have no analogues for a general reflection group $W$. 

An equivalent way to interpret the notion of transitivity is to require that the factorization cannot be realized in any \emph{proper} Young subgroup (i.e., a proper subgroup generated by transpositions) of $\Symm_n$. This interpretation makes sense for arbitrary reflection groups $W$, where we will thus say that $t_1\cdots t_k=g$ is a \defn{full reflection factorization} of an element $g\in W$ if the factors $t_i$ are reflections and they generate the \emph{full} group $W$. 

This paper is the first of a series of three \cite{DLM2, DLM3} in which we study the problem of enumerating full reflection factorizations in real and complex reflection groups -- giving formulas for what one might call the \defn{$W$-Hurwitz numbers}.  In the present paper, we focus on the infinite family $G(m, p, n)$ of ``combinatorial'' complex reflection groups; our main result provides a general formula for the (exponential) generating function for full reflection factorizations in $G(m, p, n)$ counted by length, expressed in terms of the corresponding series for the symmetric group $\Symm_n$  and the cyclic group of order $m/p$.

\begin{restatable*}[]{theorem}{mainthm}
\label{thm:main}
For an element $g\in G(m,p,n)$ with $k$ cycles, of colors $a_1, \ldots, a_k$, let $d = \gcd(a_1, \ldots, a_k, p)$. We have

\[
\FFFtr_{m,p,n}(g;z) = \frac{1}{%p^{k - 1}
    m^{n-1}} \cdot 
    \FFFtr_{m,p,1}\big(\zeta_m^{\wt(g)};n\cdot z\big) \cdot 
    \sum_{r \colon r \mid d}\left( \mu(r)\cdot r^{n + k - 2} \cdot \FFFtr_{\Symm_n}\big(\pi_{m/1}(g);(m/r)\cdot z\big) \right),
\]
where $\mu$ is the number-theoretic M\"obius function.
\end{restatable*}

(The series for the cyclic group is given in Example~\ref{ex:cyclic}.  The series for the symmetric group for the identity element is given by the Dubrovin--Yang--Zagier \cite{DYZ} recurrence -- see Remark~\ref{rem: case identity} -- and for other elements is computable using characters -- see Remark~\ref{rem: complexity}.) 

From the generating function in Theorem~\ref{thm:main} we can extract the leading coefficient, which is the number of \emph{minimum-length} full factorizations of an arbitrary element $g$ in $G(m,p,n)$. The answer in full generality (Theorem~\ref{thm:main lead coeff}) involves the usual Hurwitz numbers of genus $0$ and genus $1$ and the \emph{Jordan totient function} $J_2(m)$, which counts elements of order $m$ in the group $(\ZZ/m\ZZ)^2$. For the group $G(m,m,n)$ the leading terms are given by the following formulas.

\begin{corollary*}[of Theorem~\ref{thm:main lead coeff}]
For  $g\in G(m,m,n)$ with $k$ cycles, of lengths $\lambda_1, \ldots, \lambda_k$ and colors $a_1, \ldots, a_k$, let $d = \gcd(a_1, \ldots, a_k, m)$. We have that the number $\Ftr_{m, m, n}(g)$ of minimum-length full reflection factorizations of $g$ is
\[
\Ftr_{m,m,n}(g) = \begin{cases}
    m^{k-1} \cdot H_0(\lambda_1,\ldots,\lambda_k), & \text{ if } d = 1, \\[6pt]
    \dfrac{m^{k+1}}{d^2} J_2(d) \cdot H_1(\lambda_1,\ldots,\lambda_k), & \text{ if } d \neq 1.
    \end{cases}
\]
\end{corollary*}

In the sequel \cite{DLM2} to this paper, we will study the \emph{parabolic quasi-Coxeter elements} in a well generated complex reflection group $W$.  This wide class of elements contains the parabolic Coxeter elements, and hence all elements in the symmetric group; we will establish a number of equivalent characterizations of these elements, including in terms of the good behavior of their full factorizations.  In the climax \cite{DLM3}, we will establish a uniform  formula for the number of minimum-length full reflection factorizations of any parabolic quasi-Coxeter element in an arbitrary well generated complex reflection group that very closely models the Hurwitz formula of Theorem~\ref{thm:S_n genus 0}.

We end this introduction with a brief outline of the present paper.  In Section~\ref{sec:Background}, we introduce the main actors: we give background on complex reflection groups, with an emphasis on the combinatorial family, introduce full reflection factorizations, and compare this latter notion with the existing generalizations of transitive factorizations in the literature.  In Section~\ref{Sec: Frob lemma}, we discuss a general technique using character theory to count factorizations, and study the structural properties of their generating functions.  In Section~\ref{Sec: gen fncs combinatorially}, we use combinatorial methods to prove the main theorem (Theorem~\ref{thm:main}).  Finally, as a corollary, in Section~\ref{sec:leading terms} we extract the lowest-order coefficients $\Ftr_{m, p, n}(g)$ of the generating functions $\FFFtr_{m,p,n}(g;z)$.

\section{Reflection groups and reflection factorizations}
\label{sec:Background}

Let $V$ be a finite-dimensional complex vector space with a fixed Hermitian inner product.  A \defn{(unitary) reflection} $ t$ on $V$ is a unitary map whose \defn{fixed space} $V^{ t} := \{v \in V \colon  t(v) = v\}$ is a hyperplane (in other words, $\codim(V^{ t})=1$).  A finite subgroup $W$ of the unitary group $U(V)$ is called a \defn{complex reflection group} if it is generated by its subset $\RRR$ of reflections; we write $\mathcal{A}_W$ for its \defn{reflection arrangement}, i.e., the arrangement of fixed hyperplanes of its reflections $t\in \RRR$. We say that $W$ is \defn{irreducible} if there is no nontrivial subspace of $V$ stabilized by its action.

It is easy to see that if $W_1$ acts on $V_1$, with reflections $\RRR_1$, and $W_2$ acts on $V_2$, with reflections $\RRR_2$, then $W = W_1 \times W_2$ acts on $V = V_1 \oplus V_2$, with reflections $\RRR = (\RRR_1 \times \{\id_2\}) \cup (\{\id_1\} \times \RRR_2)$.  In this case, $W$ fails to be irreducible (as it stabilizes the subspace $V_1$ of $V$).  Shephard and Todd \cite{ShephardTodd} classified the complex reflection groups: every complex reflection group is a product of irreducibles, and every irreducible either belongs to an infinite three-parameter family $G(m,p,n)$ where $m$, $p$, and $n$ are positive integers such that $p$ divides $m$, or is one of 34 exceptional cases, numbered $G_4$ to $G_{37}$.  

\subsection{The combinatorial family $G(m,p,n)$}
\label{sec:combinatorial family}
The infinite family may be 
%concretely 
described as
%follows:
\begin{equation}\label{def:G(m, p, n)}
G(m, p, n) := \left\{ \begin{array}{c} n \times n \textrm{ monomial matrices whose nonzero entries are} \\
\textrm{$m$th roots of unity with product a $\frac{m}{p}$th root of unity} \end{array} \right\}.
\end{equation}

Writing $\zeta_m := \exp(2\pi i/ m)$ for the primitive $m$th root of unity, the reflections in $G(m, p, n)$ fall into two families: the \defn{transposition-like reflections}, of the form 
\begin{equation}
\label{eq:transposition}
\left[
\renewcommand{\arraystretch}{0.7} 
\begin{array}{c@{}c@{}c@{}cc@{}c@{}cc@{}c@{}c@{}c}
1 &&&&&&&&&& \\[-7pt]
& \ddots &&&&&&&& \\[-3pt]
&& 1 &&&&&&&& \\
&&& &&&& \zeta_m^{-k} &&& \\
&&&& 1 &&&&&& \\[-7pt]
&&&&& \ddots &&&&& \\[-3pt]
&&&&&& 1 &&&& \\
&&& \zeta_m^k &&&&&&& \\
&&&&&&&& 1 && \\[-7pt]
&&&&&&&&& \ddots & \\[-3pt]
&&&&&&&&&& 1
\end{array}
\right]
\end{equation}
for $1 \leq i < j \leq n$ and $k = 0, 1, \ldots, m - 1$, which have order $2$ and fix the hyperplanes $x_j = \zeta_m^{k} x_i$; and, when $p < m$, the \defn{diagonal reflections}, of the form
\begin{equation}
\label{eq:diagonal}
\left[
\renewcommand{\arraystretch}{0.7} 
\begin{array}{c@{}c@{}ccc@{}c@{}c}
1 &&&&&& \\[-7pt]
& \ddots &&&&& \\[-3pt]
&& 1 &&&& \\
&&& \zeta_m^{pk} &&& \\
&&&& 1 && \\[-7pt]
&&&&& \ddots & \\[-3pt]
&&&&&& 1
\end{array}
\right]
\end{equation}
for $i = 1, \ldots, n$ and $k = 1, 2, \ldots, \frac{m}{p} - 1$, which have various orders and fix the hyperplanes $x_i = 0$.

It is natural to represent such groups combinatorially. We may encode each element $w$ of $G(m, 1, n)$ by a pair $[u; a]$ with $u \in \Symm_n$ and $a = (a_1, \ldots, a_n) \in (\ZZ/m\ZZ)^n$, as follows: for $k = 1, \ldots, n$, the nonzero entry in column $k$ of $w$ is in row $u(k)$, and the value of the entry is $\zeta_m^{a_k}$.  With this encoding, it's easy to check that
\begin{equation}
\label{eq:wreath product}
[u; a] \cdot [v; b] = [uv; v(a) + b], \quad \textrm{ where } \quad v(a) := \left( a_{v(1)}, \ldots, a_{v(n)} \right).
\end{equation}
This shows that the group $G(m, 1, n)$ is isomorphic to the \defn{wreath product} $\ZZ/m\ZZ \wr \Symm_n$ of a cyclic group with the symmetric group $\Symm_n$.  If $w = [u; a]$, we say that $u$ is the \defn{underlying permutation} of $w$.

\subsubsection*{Two natural homomorphisms}
For an element $w = [u; a]$ of $G(m, 1, n)$, by a \defn{cycle} of $w$ we mean a cycle of the underlying permutation $u$.  For $S \subseteq \{1, \ldots, n\}$, we say that $\sum_{k \in S} a_k$ is the \defn{color} of $S$; this notion will come up particularly when the elements of $S$ form a cycle in $w$.  When $S = \{1, \ldots, n\}$, we call $a_1 + \ldots + a_n$ the color of the element $[u; a]$, and we denote it $\wt([u; a])$.  In this terminology, $G(m, p, n)$ is the subgroup of $G(m, 1, n)$ containing exactly those elements whose color is a multiple of $p$. 

One may think of the color as a surjective group homomorphism $\wt \colon G(m, p, n) \to p\ZZ / m \ZZ \cong G(m, p, 1)$. In fact, this is one way to see that $G(m,p',n)$ is a \emph{normal} subgroup of $G(m,p,n)$ whenever $p \mid p'$ (since $p'\ZZ/m\ZZ$ is a normal subgroup of $p\ZZ/m\ZZ$).  For any $r \mid m$, there is also another important surjective group homomorphism $G(m, 1, n) \to G(r, 1, n)$ that we define now.
\begin{definition}
For any $r \mid m$, let $\pi_{m / r} : G(m, 1, n) \to G(r, 1, n)$ be the group homomorphism defined at the level of matrices as replacing each copy of $\zeta_m$ with $\zeta_m^{m/r}$; equivalently, viewing $G(r, 1, n)$ as a subgroup of $G(m, 1, n)$, it is the map 
\[
\pi_{m / r}([u; a]) = \left[u;  \frac{m}{r} \cdot a \right].
\]
\end{definition}
When restricted to the subgroup $G(m, p, n)$, the map $\pi_{m / p}$ is a surjective homomorphism onto $G(p, p, n)$.  In particular, for any $[u; a] \in G(m, p, n)$, applying $\pi_{m/1}$ recovers the underlying permutation $u \in \Symm_n$ of $w$.

Conjugacy classes in $G(m, 1, n)$ are commonly indexed by tuples $(\lambda_0, \ldots, \lambda_{m - 1})$ of integer partitions of total size $n$.  In this indexing, $\lambda_j$ is the partition composed of the lengths of the cycles of color $j$.  In other words, two elements in $G(m, 1, n)$ are conjugate if and only if they have the same number of cycles of length $k$ and color $c$ for all $k = 1, \ldots, n$, $c \in \ZZ/m\ZZ$.

\subsubsection*{Graph interpretation and connected subgroups}
It is natural to represent collections of reflections in $G(m, p, n)$ as graphs on the vertex set $\{1, \ldots, n\}$: a transposition-like reflection whose underlying permutation is $(i j)$ can be represented by an edge joining $i$ to $j$, while a diagonal reflection whose nonzero color occurs at position $i$ can be represented by a loop at vertex $i$.

We say that a collection of reflections in $G(m, p, n)$ is \defn{connected} if the associated graph is connected.  It is a basic fact of graph theory that any connected graph contains a spanning tree; consequently, any connected set of reflections in $G(m, p, n)$ contains a connected subset of $n - 1$ transposition-like reflections.  Such a subset always generates a subgroup of $G(m, p, n)$ isomorphic to the symmetric group $\Symm_n$ \cite[Lem.~2.7]{Shi2005}. 

More generally, if $H$ is a subgroup of $G(m, p, n)$ generated by a connected collection of reflections, then $H$ is isomorphic to $G(m', p', n)$ for some integers $m', p'$ such that $m' \mid m$, $p' \mid m'$, and $\frac{m'}{p'} \mid \frac{m}{p}$.  The isomorphism can be concretely realized as conjugation by an appropriate diagonal element of $G(m, 1, n)$.

\subsection{Reflection factorizations and reflection length}

Starting with an arbitrary group $G$ and some generating set $S\subset G$,  the Cayley graph of $G$ with respect to $S$ determines a natural length function on the elements of the group. Indeed, we may define the length of an element $g\in G$ as the size of the shortest path in the Cayley graph from the identity element to $g$.

In the case of a complex reflection group $W$, we pick the set $\RRR$ of reflections as the distinguished generating set.  Then the \defn{reflection length} $\lR(g)$ of an element $g\in W$ is the smallest number $k$ such that there exist reflections $t_1, \ldots, t_k$ that form a factorization $g=t_1\cdots t_k$.  (Here the superscript ``red'' stands for ``reduced''.)

For an arbitrary number $k$, if $t_1,\ldots,t_k$ are reflections such that $t_1\cdots t_k = g$, we say that the tuple $( t_1,\ldots, t_k)$ forms a \defn{reflection factorization} of $g$ of length $k$.

Say that a reflection factorization $t_1 \cdots t_k = g$ of an element $g \in W$ is \defn{full} (relative to $W$) if the factors generate the full group, i.e., if $W = \langle t_1, \ldots, t_k \rangle$.  Observe that, by definition, every reflection factorization of any element $g$ is full \emph{relative to the subgroup generated by the factors}.  
The \defn{full reflection length} $\ltr(g)$ of $g$ is the minimum length of a full reflection factorization:
\[
\ltr(g) := \min \Big\{ k : \exists\ t_1, \ldots, t_k \in \RRR \text{ such that } t_1 \cdots t_k = g \text{ and } \langle t_1, \ldots, t_k \rangle = W \Big\}.
\]

As we wish to enumerate factorizations of arbitrary length $N$, it is convenient to encode the answers via generating functions.  We denote by $\FFF_W(g;z)$ the exponential generating function for the number of reflection factorizations of $g$ of arbitrary length:
\begin{equation}
\label{Eq: defn F_W(g;t)}
\FFF_W(g;z):=\sum_{N\geq 0}\#\left\{(t_1,\ldots,t_N)\in\RRR^N: t_1\cdots t_N=g\right\}\cdot \frac{z^N}{N!}.
\end{equation}
By definition, the lowest-order term of $\FFF_W(g;z)$ occurs in degree $\lR(g)$.

Similarly, we denote by $\FFFtr_W(g;t)$ the exponential generating function for the number of \emph{full} reflection factorizations of $g$ of arbitrary length:
\begin{equation}
\label{Eq: defn Ftr_W(g;t)}
\FFFtr_W(g;z):=\sum_{N\geq 0}\#\left\{(t_1,\ldots,t_N)\in\RRR^N : t_1\cdots t_N=g\text{ and }\langle t_1,\ldots, t_N\rangle=W \right\}\cdot \frac{z^N}{N!}.
\end{equation}
We further write $\Ftr_W(g)$ for the number of minimum-length full factorizations (i.e., with $N=\ltr(g)$), so that the lowest-order term of $\FFF_W(g;t)$ is equal to $\Ftr_W(g)\cdot z^{\ltr(g)}/\ltr(g)!$.

\subsection{Transitive factorizations} \label{sec: transitivity}

In the symmetric group $\Symm_n$, it is of particular interest to study \defn{transitive} factorizations, that is, those for which the factors generate a subgroup that acts transitively on $\{1, \ldots, n\}$ (see, e.g., \cite{Hurwitz, GJ97, B-MS} and the survey \cite{GJSurvey}).  The reflections in $\Symm_n$ are the transpositions, and it is easy to see that a transposition factorization is transitive if and only if it is connected, if and only if it is full.

There have been several attempts to generalize the notion of transitivity to reflection groups other than $\Symm_n$.  In \cite{BGJ}, the authors consider the hyperoctahedral group $G(2, 1, n)$ (the Coxeter group of type $B_n$) in its natural permutation action on $E_2 := \{\pm e_1, \ldots, \pm e_n\}$, where $e_i$ is a standard basis vector.  They enumerate transitive reflection factorizations of arbitrary elements under two notions of transitivity: a strong version (that they call \emph{admissibility}), in which the factors act transitively on $E_2$, and a weaker version (that they call \emph{near-admissibility}), in which the factors merely act transitively on the coordinate axes (equivalently, on the pairs $\{\{\pm e_i\} \colon i = 1, \ldots, n\}$).

Both notions of transitivity considered by \cite{BGJ} have been recently studied in the infinite family $G(m, p, n)$ of complex reflection groups.  It is easy to see that a reflection factorization in $G(m, p, n)$ is transitive in the weaker sense (i.e., that the group generated by the factors acts transitively on the coordinate axes in $\CC^n$) if and only if the factorization is connected.  In \cite{PR}, the authors consider connected reflection factorizations of arbitrary elements in $G(m, p, n)$, and give an analogue of the polynomiality property (see \cite[\S 1.1]{GJV-dH} and references therein) of the usual Hurwitz numbers  (famous due to its connection with the ELSV formula \cite{ELSV}).  And in \cite{LM}, the authors enumerate factorizations of a Coxeter element in $G(m, 1, n)$ or $G(m, m, n)$ into arbitrary factors that are transitive in the stronger sense of the action on $E_m := \{\zeta_m^k e_j \colon k = 0, \ldots, m - 1, j = 1, \ldots, n\}$.

The main drawback of results based on these definitions of transitivity is that it is unclear how to generalize them to reflection groups outside the infinite family (see, e.g., \cite[Ques.~8.2]{LM}).  Indeed, the groups $G(m, p, n)$ are precisely the \defn{imprimitive} irreducible complex reflection groups, i.e., the groups $W$ for which there is a decomposition $V = V_1 \oplus \cdots \oplus V_n$ of the vector space on which $W$ acts that is respected by the action of $W$; for other (\emph{primitive}) groups, there is no natural set analogous to $\{1, \ldots, n\}$ for $\Symm_n$.  Thus, one thesis of the present work and its sequels is that fullness is the ``correct'' generalization of transitivity from $\Symm_n$ to reflection groups because it allows one to pose questions uniformly.\footnote{For factorizations whose factors are not necessarily reflections, the notion of fullness would require that there is no \emph{proper} reflection subgroup containing all the factors.}  And, moreover, the main result of \cite{DLM3} will establish that such questions may have uniform answers.
Before we move on to the main result of the present paper, we offer one final point in support of this thesis: it is not only in $\Symm_n$ but also in $G(m,m,n)$ that there is a coincidence between factorizations that are full and those that are transitive (in an appropriate sense).

\begin{proposition}
A set of reflections in $W = G(m, m, n)$ acts transitively on the set $E_m = \{\zeta_m^k e_j \colon k = 0, \ldots, m - 1, j = 1, \ldots, n\}$ (where $e_j$ is a standard basis vector of $\CC^n$ and $\zeta_m$ is the primitive $m$th root of unity) if and only if it generates the full group $W$.
\end{proposition}
\begin{proof}
Any set of reflections of $W = G(m, m, n)$ generates a reflection subgroup $W'$ of $W$.  If the set is not connected, then $W'$ does not act transitively on the coordinate axes of $\CC^n$, let alone on the set $E_m$.  Therefore we restrict our attention to connected sets of reflections.  It follows from the observations in the last paragraph of Section~\ref{sec:combinatorial family} that every connected subgroup of $G(m, m, n)$ is conjugate by an element of $G(m, 1, n)$ to a subgroup of the form $G(p, p, n)$ for some $p \mid m$.  Conjugation by elements of $G(m, 1, n)$ preserves both the properties of transitivity and fullness, so it is enough to observe that $G(p, p, n)$ acts transitively on $E_m$ if and only if $p = m$.
\end{proof}

Of course, this equivalence does not extend to $G(m, p, n)$ for $p < m$, since each such group contains the proper reflection subgroup $G(m, m, n)$.

\section{A general approach to counting via the Frobenius lemma}
\label{Sec: Frob lemma}

In the previous section we defined, in Equations \eqref{Eq: defn F_W(g;t)} and \eqref{Eq: defn Ftr_W(g;t)}, exponential generating functions $\FFF_W(g;z)$ and $\FFFtr_W(g;z)$ that enumerate factorizations whose factors are taken from the set of reflections $\RRR$ of the complex reflection group $W$. Because $\RRR$ is closed under conjugation, a traditional approach via representation theory, due originally to Frobenius \cite{Frobenius}, allows us to express both series as a finite sum of exponentials in $z$. We start with the case of arbitrary (i.e., not necessarily full) reflection factorizations. Following \cite[(4.3)]{CC}, we have the formula  
\begin{equation}
\label{eq: frobenius no mobius}
\FFF_W(g;z) = \dfrac{1}{\# W} \cdot \sum_{\chi\in\widehat{W}} \chi(1) \cdot \chi(g^{-1})\cdot\exp\left(\dfrac{\chi(\RRR)}{\chi(1)}\cdot z\right),
\end{equation}
where
$\widehat{W}$ denotes the set of irreducible complex characters of $W$ and $\chi(\RRR):=\sum_{t\in\RRR}\chi(t)$. For any reflection group $W$, the normalized character values $\chi(\RRR)/\chi(1)$ are integers (see \cite[Def.~3.14 and Prop.~3.15]{chapuy_theo}) and thus $\FFF_W(g;z)$ can be written as a Laurent polynomial in the variable $X:=e^z$. 

\subsection{Full factorizations after an inclusion-exclusion argument}

Any factorization $(t_1, \ldots, t_N)$ is full for the subgroup $W':=\langle t_1,\dots,t_N\rangle$ generated by its
elements.  Therefore, we may view the series $\FFF_W(g;z)$ in \eqref{Eq: defn F_W(g;t)} as the sum, over all reflection subgroups $W'$ of $W$ that contain $g$, of the generating functions $\FFFtr_{W'}(g;z)$ of full factorizations of $g$ with respect to $W'$. By applying the principle of inclusion-exclusion, we conclude that
\begin{equation} 
\label{eq: mobius transitive}
\FFFtr_W(g;z)=\sum_{g\in W'\leq W}\mu(W,W')\cdot\FFF_{W'}(g;z),
\end{equation} 
where the M\"obius function is computed in the poset of all reflection subgroups $W'$ of $W$, ordered by reverse inclusion. This is analogous to the usual construction \cite[\S5.2]{EC2} of taking the logarithm of an exponential generating function to count ``connected components", which was also what led Goulden and Jackson to rediscover the Hurwitz formula of Theorem~\ref{thm:S_n genus 0}, see \cite[\S4]{GJ97}. Indeed, in the case of the symmetric group $\Symm_n$, the poset of reflection subgroups $W'\leq W$ is isomorphic to the set partition lattice of $\{1, \ldots, n\}$, and M\"obius inversion corresponds precisely to taking the logarithm of generating functions.

The irreducible characters for complex reflection groups $W$ can be constructed via the computer software SageMath and CHEVIE \cite{sagemath, chevie}, which realize reflection groups
via their permutation action on roots. This allows for explicit calculation of the functions $\FFF_W(g;z)$ and $\FFFtr_W(g;z)$ for the exceptional groups $G_{4}$ to $G_{36}$.  For the interested reader, the series for the case $g = \id$ and all exceptional groups may be found in Appendix~\ref{appendix}; they are also attached to this arXiv submission as an auxiliary file.  (For the reflection group $G_{37}=E_8$, the lattice of reflection subgroups and its M\"obius function are very complicated, so computation of $\FFFtr_{E_8}(g; z)$ for arbitrary $g$ requires significant computational resources.) 

\begin{example}
\label{ex:cyclic}
For example, consider the case of an arbitrary element $g$ in the cyclic group $W = G(m, 1, 1)$ with $m > 1$, in which all non-identity elements are reflections.  It is easy to show (either by representation theory or by an elementary inductive argument) that
\[
\FFF_{W}(g; \log X) = \begin{cases}
    \dfrac{X^m - 1}{mX}, & g \neq \id \\[12pt]
    \dfrac{X^m + (m - 1)}{mX}, & g = \id.
    \end{cases}
\]

The subgroups of $W$ are the cyclic groups $G(m, r, 1) \cong \ZZ/(m/r)\ZZ$ for $r \mid m$, and their poset of inclusions is isomorphic to the poset of positive integer divisors of $m$.  Furthermore, an element $g$ of $G(m, 1, 1)$ belongs to $G(m, r, 1)$ if and only if its order $\op{ord}(g)$ divides $r$.  Therefore, doing the M\"obius inversion over all subgroups, we have (after some simplification) that

\[
\FFFtr_W(g; \log X) = \frac{1}{X} \sum_{\substack{r \colon r \mid m \\ \text{and} \op{ord}(g) \mid r}} \mu(m/r) \frac{X^r - 1}{r} = \frac{X - 1}{X} \sum_{\substack{r \colon r \mid m \\ \text{and} \op{ord}(g) \mid r}} \mu(m/r) \frac{[r]_X}{r}
\]
for all $g$ in $W$, where $[r]_X := 1 + X + \ldots + X^{r - 1}$.
\end{example}

Another example, in the case of the real reflection group $H_3 = G_{23}$, is given below as Example~\ref{Example: FFFtr(H3)}.

\begin{remark}
\label{Rem: length via Frobenius}
The same representation-theoretic approach can easily be used to obtain the reflection length and full reflection length of any element $g\in W$. Indeed, as mentioned below Equations \eqref{Eq: defn F_W(g;t)} and \eqref{Eq: defn Ftr_W(g;t)}, one need only read off the lowest-order terms of the generating functions $\FFF_W(g;z)$ and $\FFFtr_W(g;z)$, respectively. \end{remark}

\begin{remark} \label{rem: complexity}
The two formulas \eqref{eq: frobenius no mobius} and \eqref{eq: mobius transitive} imply that the generating functions $\FFFtr_W(g;z)$ are finite sums of exponentials $e^{kz}$. In fact, there are much stronger constraints on their structure. As we will see in Proposition~\ref{Prop: structural result FFFtr} in the next section, the numbers $k$ appearing in the exponents are integers and belong to an interval of length at most $h\cdot n$, where $h$ is the Coxeter number and $n$ the rank of the irreducible group $W$. 

We warn the reader that several relevant results in factorization enumeration (e.g., \cite{GJ97,GJ99,GJV-dH}) concern generating functions that vary the index $n$ of the symmetric group $\Symm_n$ but fix the genus, whereas we fix the group but vary the genus (equivalently, the number of factors). These generating functions, though similar to ours at first glance, have much more complicated structure and asymptotics. 
\end{remark}

\subsection{Additional structure for the generating functions}

We give now some further properties of the functions $\FFFtr_W(g;z)$ that are not necessary in what follows, but that might shed some light on the structural characteristics of our main theorems. 

\begin{proposition}\label{Prop: structural result FFFtr}
Let $W$ be any complex reflection group and $g$ any element of $W$.
The exponential generating function $\FFFtr_W(g;z)$ of full reflection factorizations of $g$ can always be expressed in terms of the variable $X:=e^z$ as a Laurent polynomial of the form
\[
\FFFtr_W(g; \log X)=\frac{1}{\#W} \cdot \Phi_{W}(g;X) \cdot (X-1)^{\ltr(g)} \cdot \frac{1}{X^{\#\mathcal{A}_W}},
\]
with $\Phi_{W}(g;X)$ a monic polynomial in $X$ of degree $h\cdot n-\ltr(g)$ and $\mathcal{A}_W$ the reflection arrangement of $W$.
\end{proposition}

\begin{proof}
As mentioned in the first paragraph of Section~\ref{Sec: Frob lemma}, the generating function $\FFFtr_W(g;z)$ is always a Laurent polynomial in terms of the variable $X:=e^z$. Since $e^z = 1 + z + \ldots$, the multiplicity of the root $X=1$ in this polynomial agrees with the degree $\ltr(g)$ of the lowest-order term in the $z$-expansion of the series. The highest and lowest degree of this Laurent polynomial are determined by the inequalities of \cite[Prop.~3.2]{D2} regarding the normalized traces $\chi(\mathcal{R})/\chi(1)$; it is monic because only the trivial representation $\chi_{\op{triv}}$ has normalized trace equal to $|\mathcal{R}|$ (as in the proof of \cite[Thm.~3.6]{D2}). 
\end{proof}

\begin{example}[Counting full reflection factorizations of the identity element in $H_3$]\label{Example: FFFtr(H3)}

We give here in some detail a snapshot of the calculation of the generating function $\FFFtr_{H_3}(\id;z)$ counting full factorizations of the identity element in the real reflection group $H_3 = G_{23}$. M\"obius inversion on the poset of reflection subgroups of a reflection group $W$ can be computed recursively, as follows: assume that we have calculated in all reflection subgroups $W'\lneq W$ the corresponding generating functions $\FFFtr_{W'}(g;z)$ of reflection factorizations of $g$ that are full in $W'$; then $\FFFtr_W(g; z)$ is the result of subtracting all of these from the generating function $\FFF_W(g;z)$ of \emph{not necessarily full} reflection factorizations in $W$.

In Table~\ref{Table: example H_3} below we see this information for the reflection group $H_3$. Up to conjugacy, it has six classes of proper reflection subgroups (see also \cite[\S6]{DPR}), including the trivial subgroup $\{\id\}$. The first column records the Coxeter type of the classes, the second column records the number of different subgroups in each class, and the last column displays the already-calculated generating functions. Notice that we have given them in terms of the variable $X:=e^z$ so that one can immediately see the structural properties of Proposition~\ref{Prop: structural result FFFtr}. In particular, the exponent of the factor $X-1$ is the full reflection length of the identity in each subgroup $W'$.

\begin{table}[H]
\renewcommand{\arraystretch}{1.5}
\[
\begin{array}{|l|c|c|}
\hline
\text{Type of }W' & \# \big[W'\big] & \FFFtr_{W'}(\id; \log X) \rule{0em}{15pt}\\[3pt]
\hline\hline
\{\id\} & 1 & 1 \rule{0em}{15pt}\\[3pt]
\hline 
A_1 & 15 & \dfrac{1}{2}\cdot (X - 1)^2\cdot\dfrac{1}{X}  \rule{0em}{18pt}\\[6pt]
\hline
A_1^2 & 15 & \dfrac{1}{4}\cdot (X - 1)^4\cdot\dfrac{1}{X^2} \rule{0em}{18pt}\\[6pt]
\hline
A_2 & 10 & \dfrac{1}{6}\cdot (X^2 + 4X + 1)\cdot (X - 1)^4\cdot\dfrac{1}{X^3} \rule{0em}{18pt}\\[6pt]
\hline
I_2(5) & 6 & \dfrac{1}{10}\cdot (X^6 + 4X^5 + 10X^4 + 20X^3 + 10X^2 + 4X + 1)\cdot (X - 1)^4\cdot\dfrac{1}{X^5} \rule{0em}{18pt}\\[6pt]
\hline 
A_1^3 & 5 & \dfrac{1}{8}\cdot (X-1)^6\cdot \dfrac{1}{X^3} \rule{0em}{18pt}\\[6pt]
\hline 
\end{array}
\]
\caption{The generating functions $\FFFtr_{W'}(\id; z)$ for all proper reflection subgroups $W'$ of $H_3$, expressed as Laurent polynomials in $X = e^z$.}
\label{Table: example H_3}
\renewcommand{\arraystretch}{1}
\end{table} 

Now, a direct computer calculation as in \eqref{eq: frobenius no mobius} gives that the generating function for not-necessarily-full reflection factorizations of the identity is
\[
\FFF_{H_3}(\id; \log X)=\dfrac{1}{120}\cdot (X^{30} + 18X^{20} + 25X^{18} + 32X^{15} + 25X^{12} + 18X^{10} + 1)\cdot \dfrac{1}{X^{15}}.
\]
We can finally write
\begin{alignat*}{1}
\FFFtr_{H_3}(\id; \log X) & =\FFF_{H_3}(\id; \log X)-\sum_{W'\lneq W}\FFFtr_{W'}(\id; \log X) \\
    &=\dfrac{1}{120}\cdot \Big(X^{24} + 6X^{23} + 21X^{22} + 56X^{21} + 126X^{20} + 252X^{19}  + 462X^{18} \\
    & \quad + 792X^{17} + 1287X^{16}  + 2002X^{15} + 2949X^{14} + 4044X^{13} \\
    & \quad + 4804X^{12} + 4044X^{11} + 2949X^{10} + 2002X^9 + 1287X^8 + 792X^7 \\ 
    & \quad + 462X^6 + 252X^5 + 126X^4 + 56X^3 + 21X^2 + 6X + 1\Big)\cdot(X - 1)^6\cdot \dfrac{1}{X^{15}}.
\end{alignat*}
\end{example}

\begin{remark} \label{rem: positive unimodal Phi}
The polynomials $\Phi_{W}(g;X)$ that appear in Proposition~\ref{Prop: structural result FFFtr} are very interesting: in real reflection groups, it is easy to show that they are palindromic, but furthermore it seems  (based on empirical data) that they have positive integer coefficients; in most cases (but not always, see $G_2=I_2(6)=G(6,6,2)$) their coefficients are unimodal, and their roots have an interesting structure (see Figure~\ref{fig:roots} and Appendix~\ref{appendix}). Neither palindromicity nor positivity hold for all complex types; for instance, in the cyclic group $G(6,1,1)\cong \ZZ/6\ZZ$ we have
\[
\Phi_{G(6,1,1)}(\id;X)=X^4 + 2X^3 + 3X^2 + 2X - 2.
\]

\end{remark}

\begin{figure}
    \centering
    \includegraphics[scale=0.5]{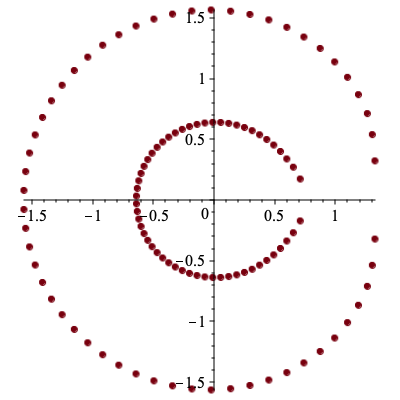}
    \quad
    \includegraphics[scale=0.5]{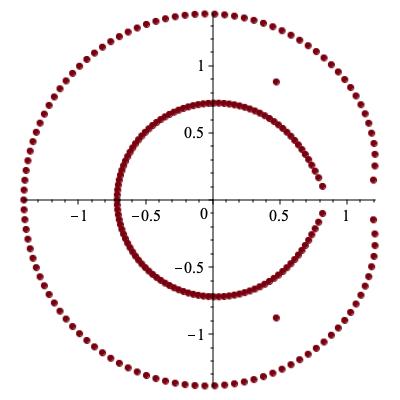}
    \caption{Roots of the polynomials $\Phi_{E_7}(\id; X)$ and $\Phi_{E_8}(\id; X)$. The roots come in pairs $(a,1/a)$ because the polynomials are palindromic, and this explains the two ``components" in the figures, but we cannot a priori justify the rest of the structure (i.e., the orderly arrangement of the roots).} 
    \label{fig:roots}
\end{figure}

\begin{remark} \label{rem: case identity}
While the machinery of the Frobenius lemma gives a powerful tool to study factorizations in any particular group, the answers it produces for an infinite family of groups may not be usable.  For example, in the symmetric group, one can produce explicit formulas for the number of genus-$0$ and genus-$1$ factorizations of an arbitrary element, but such formulas are unknown for the whole generating series $\FFFtr_{\Symm_n}(g; z)$.  However, there is one case where some structure is indeed known: in the case of the identity element, Dubrovin--Yang--Zagier \cite{DYZ} proved the following recursion (which, as usual, we phrase in terms of the parameter $X=e^z$):
\begin{multline} \label{eq:Dubrovin--Yang--Zagier}
n^2(n-1)\cdot \FFFtr_{\Symm_n}(\id;\log X) = \\ \sum_{k=1}^{n-1}k(n-k)^2\cdot\binom{n}{k}\cdot(X^k-2+X^{-k})\cdot\FFFtr_{\Symm_k}(\id; \log X)\cdot\FFFtr_{\Symm_{n-k}}(\id; \log X).
\end{multline}
This allows one to compute the function $\FFFtr_{\Symm_n}(\id; z)$ efficiently.
\end{remark}

\begin{remark}
From the recurrence \eqref{eq:Dubrovin--Yang--Zagier}, the polynomials $\Phi_{\Symm_n}(\id, X)$
 have integer coefficients and are palindromic. This forces their roots to come in pairs $(a, 1/a)$ and $(a,\overline{a})$ and thus the plots of the roots have two components corresponding to the $(a,1/a)$ pairs. Figure~\ref{fig:limit} shows plots of the roots  of $\Phi_{\Symm_{2n}}(\id, X)$ for $n=5,\ldots,25$. Are these points converging to a limit shape as $n\to \infty$ (for example, to the unit circle)? 
\end{remark}

\begin{figure}
    \centering
    \includegraphics[scale=0.4]{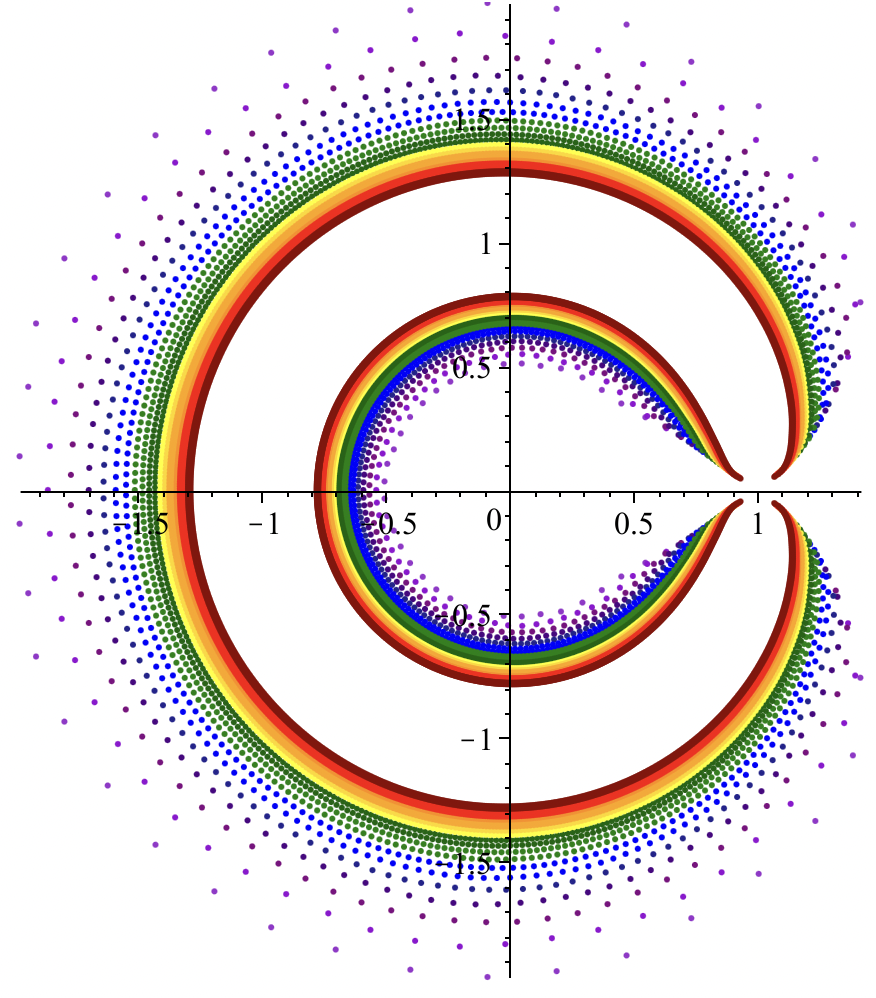}
    \caption{Plot of the roots of $\Phi_{\Symm_{2n}}(\id, X)$ for $n=\textcolor{violet}{5}\textcolor{purple}{,} \, \textcolor{blue}{.}\,\textcolor{green}{.}\,\textcolor{yellow}{.}\,\textcolor{orange}{,} \, \textcolor{red}{25}$. (We skip the roots for odd $n$ to make the plot clearer.)}
    \label{fig:limit}
\end{figure}

\section{Counting full factorizations in the combinatorial family} 
\label{Sec: counting full fancs}
\label{Sec: gen fncs combinatorially}

In this section, we give a formula for the generating function that counts full reflection factorizations of arbitrary length (not just minimum length) of an element $g$ in a group $G(m, p, n)$ from the infinite family.  Our formula is valid for arbitrary $g$ and expresses the result in terms of similar generating functions in the symmetric group $\Symm_n$ and in the cyclic group $G(m, p, 1) \cong \ZZ/(m/p)\ZZ$. To avoid the cumbersome notation $\FFFtr_{G(m,p,n)}(g;z)$ we write the generating functions for the infinite family simply as $\FFFtr_{m,p,n}(g;z)$.

\mainthm
(Since $p \mid m$ and any two integer representatives of $a_i \in \ZZ/m\ZZ$ differ by a multiple of $m$, the number $d$ is well defined as an integer, and is always a divisor of $p$.)

\begin{remark}
\label{rem: main theorem special cases}
In the case that $p = 1$, then always $d = 1$ and so the summation simplifies to a single term.
In the case that $p = m$, the group $G(m, m, 1)$ is the trivial group, which is still a reflection group but with an empty set of reflections $\RRR=\emptyset$; then by construction (see \eqref{eq: frobenius no mobius} and \eqref{eq: mobius transitive}) we will have that $\FFFtr_{m, m, 1}(\id;z)=1$.
\end{remark}

Our approach is broadly similar to that of \cite{LM} and Polak--Ross \cite{PR}; we use the maps $\pi$ to project factorizations from $G(m, p, n)$ to a simpler subgroup (either $G(p, p, n)$ or $\Symm_n$) and then compute the size of each fiber of the projection.  Our main tool for counting preimages (Lemma~\ref{tree lemma})  is adapted directly from \cite{PR}; however, because we consider full (rather than connected) factorizations, we require an additional inclusion-exclusion argument over a certain subgroup lattice.

The structure of the proof is as follows: first, we characterize generating sets of reflections in $G(m, p, n)$ in terms of their projections into $G(m, p, 1)$ and into $G(p, p, n)$.  Second, we use this characterization to express $\FFFtr_{m, p, n}(g; z)$ in terms of related series in the cyclic group $G(m, p, 1)$ and the color-$0$ subgroup $G(p, p, n)$.  Finally, we show how to express $\FFFtr_{p, p, n}(g; z)$ in terms of related series in the symmetric group $\Symm_n$.

\subsection{Characterization of generating sets}

We begin by giving a characterization for generating sets of reflections in $G(m, p, n)$ in terms of their projections.

\begin{lemma} 
\label{generating sets lemma}
A set $S$ of reflections in $G(m, p, n)$ generates the whole group if and only if $\wt(S)$ generates $p\ZZ / m \ZZ$ and $\pi_{m/p}(S)$ generates $G(p, p, n)$. 
\end{lemma}
\begin{proof}
If $S$ generates $G(m, p, n)$ then $\pi_{m/p}(S)$ generates $\pi_{m/p}(G(m, p, n)) = G(p, p, n)$ and $\wt(S)$ generates $\wt(G(m, p, n)) = p\ZZ/m\ZZ$ because both projections are surjective.

Conversely, suppose $S$ is a set of reflections in $G(m, p, n)$ such that $\wt(S)$ generates $p\ZZ / m \ZZ$ and $\pi_{m/p}(S)$ generates $G(p, p, n)$.  Write $S$ as a disjoint union $S = S_{\mathrm{tr}} \cup S_{\mathrm{diag}}$, where $S_{\mathrm{tr}}$ consists of transposition-like reflections and $S_{\mathrm{diag}}$ consists of diagonal reflections.
Consider any element $w = [u; (a_1, \ldots, a_n)]$ in $G(m, p, n)$ (so $p \mid a_1 + \ldots + a_n$).  By definition, we have $\pi_{m/p}(g) = [u; (m/p \cdot a_1, \ldots, m/p \cdot a_n)] \in G(p, p, n)$.  By hypothesis, we can write
\[
\pi_{m/p}(g) = \pi_{m/p}(s_1) \cdots \pi_{m/p}(s_k)
\]
for some reflections $s_1, \ldots, s_k$ in $S$.  Then 
\[
g \cdot (s_1 \cdots s_k)^{-1} = [ \id; (b_1, \ldots, b_k)]
\]
belongs to the kernel of $\pi_{m/p}$, that is, $p$ divides $b_i$ for $i = 1, \ldots, n$.  To finish, it suffices to show that this element $ [ \id; (b_1, \ldots, b_k)]$ belongs to the subgroup $\langle S \rangle$ generated by $S$.
Since $\pi_{m/p}(S)$ generates $G(p, p, n) \supseteq \Symm_n$, it must be that $S_{\mathrm{tr}}$ is connected, and so products of elements of $S_{\mathrm{tr}}$ yield arbitrary underlying permutations.  Conjugating a diagonal reflection by such elements produces the diagonal reflections with the same color and arbitrary position for the nontrivial entry, and hence the subgroup
\[
\Big\langle \; 
\langle S\rangle\text{-conjugates of } S_{\mathrm{diag}}
\; \Big\rangle
\]
generated by $S_{\mathrm{diag}}$ and all of its conjugates by products of elements of $S$ is equal to the diagonal subgroup
\[
%G(m, p, 1)^n = 
\Big\{ \; 
[\id; (a_1, \ldots, a_n)] \colon p \textrm{ divides } a_1 + \ldots + a_n 
\;\Big\}.
\]
This subgroup manifestly contains $\ker(\pi_{m/p})$.  Thus $g \in \langle S \rangle$.  Since $g$ was arbitrary, $\langle S \rangle = G(m, p, n)$, as claimed.
\end{proof}

\begin{lemma}
\label{tree lemma}
Let $(t_1, \ldots, t_\ell)$ be a sequence of permutations in $\Symm_n$ with the property that its subsequence of transpositions is connected, and let $g = t_1 \cdots t_\ell$.  There exists a subsequence $(t_{i_1}, \ldots, t_{i_{n - 1}})$ of $n - 1$ transpositions with the following properties: for every choice of $\{\widetilde{t}_j \colon j \neq i_1, \ldots, i_{n - 1}\}$ and $\widetilde{g}\in G(m, 1, n)$ such that 
\begin{itemize}
\item $\pi_{m/1}(\widetilde{t}_j) = t_j$ for $j \neq i_1, \ldots, i_{n - 1}$,
\item $\pi_{m/1}(\widetilde{g}) = g$, and
\item $\displaystyle \wt(\widetilde{g}) = \sum_{j \neq i_1, \ldots, i_{n - 1}} \wt(\widetilde{t}_j)$,
\end{itemize}
there exists a unique choice of reflections $\widetilde{t}_{i_1}, \ldots, \widetilde{t}_{i_{n - 1}}$ in $G(m, 1, n)$ such that $\pi_{m/1}(\widetilde{t_{i_j}}) = t_{i_j}$ and $\widetilde{g} = \widetilde{t}_{1} \cdots \widetilde{t}_{\ell}$.
\end{lemma}
\begin{proof}
This is essentially Lemma 3.2 and the argument immediately following it in \cite{PR}. 
\end{proof}

\subsection{From $G(m, p, n)$ to $G(p, p, n)$}

Next, we show how to express the generating function for full factorizations of an element $g \in G(m, p, n)$ in terms of the series of $G(p, p, n)$-full factorizations of its projection $\pi_{m/p}(g)$. 

\begin{proposition}
\label{thm:gf for arbitrary elements in G(m, p, n)}
Suppose $p < m$.  Then for any element $g \in G(m, p, n)$, we have
\[
\FFFtr_{m, p, n}(g; z) = \dfrac{1}{(m/p)^{n-1}}  \cdot \FFFtr_{p, p, n}( \pi_{m / p}(g); (m/p)\cdot z) \cdot \FFFtr_{m, p, 1}(\zeta_m^{\wt(g)}; n \cdot z).
\]
\end{proposition}
\begin{proof}
Let $W = G(m, p, n)$ with $p < m$ and let $g$ be an arbitrary element of $W$.  For each reflection factorization $f$ of $g$, let $f_1$ be the result of deleting all copies of the identity from $\pi_{m/p}(f)$ (equivalently, the result of applying $\pi_{m/p}$ only to the transposition-like factors in $f$) and let $f_2$ be the result of deleting all copies of the identity from $\zeta_m^{\wt(f)}$  (equivalently, of applying $\zeta_m^{\wt}$ only to the diagonal factors in $f$).  Then $f_1$ is a $G(p, p, n)$-reflection factorization of $\pi_{m/p}(g)$ and $f_2$ is a $G(m, p, 1)$-reflection factorization of $\zeta_m^{\wt(g)}$, and their lengths sum to the length of $f$.  By Lemma~\ref{generating sets lemma}, $f$ generates $W$ if and only if $f_1$ generates $G(p, p, n)$ and $f_2$ generates $G(m, p, 1)$.  We now fix such a pair $(f_1, f_2)$ and consider how many preimages $f$ it has among the reflection factorizations of $g$.

Let $k = \# f_1$ and $\ell - k = \# f_2$.  Then the number of preimages of $(f_1, f_2)$ among factorizations of \emph{all} elements in $W$ is $\binom{\ell}{k} \cdot (m / p)^{k} \cdot n^{\ell - k}$: the binomial coefficient counts the ways to assign $k$ positions to transposition-like factors from among $\ell$ positions, the factor $(m / p)^k$ counts the ways to choose for each of $k$ elements of $f_1$ a reflection preimage under $\pi_{m / p}$, and the factor $n^{\ell - k}$ counts the ways to choose for each of $\ell - k$ elements of $f_2$ a reflection preimage under $\zeta_m^{\wt}$.

We claim that among these preimages, precisely $\binom{\ell}{k} \cdot (m / p)^{k - (n - 1)} \cdot n^{\ell - k}$ are factorizations of $g$.  Indeed, following the first step of the construction in the previous paragraph, choose a length-$\ell$ factorization $f^*$ of $\pi_{m/p}(g)$ by shuffling $f_1$ with $\ell - k$ copies of the identity.  
Now project $f^*$ to the symmetric group, giving us an $\Symm_n$-factorization $\pi_{m/1}(f^*)$ of $\pi_{m/1}(g)$.
By construction, $\pi_{m/1}(f_1)$ is equal to the subsequence of transpositions in $\pi_{m/1}(f^*)$.  Since $f_1$ generates $G(p, p, n)$, it is connected, and so $\pi_{m/1}(f_1)$ is connected as well.  Therefore, Lemma~\ref{tree lemma} provides a special subsequence $t_{i_1}, \ldots, t_{i_{n - 1}}$ of $\pi_{m/1}(f^*)$.  As in the previous paragraph, there are $n^{\ell - k} \cdot (m/p)^{k - (n - 1)}$ ways to choose reflection preimages $\widetilde{t}_j$ in $G(m, p, n)$ of the factors $t_j$ in $f^*$ for $j \neq i_1, \ldots, i_{n - 1}$, consistent with the restrictions that the projections under $\pi_{m/p}$ and $\zeta_m^{\wt}$ should give $f_1$ and $f_2$.  By Lemma~\ref{tree lemma}, there is a unique way to lift $\pi_{m/1}(t_{i_1})$, \ldots, $\pi_{m/1}(t_{i_{n - 1}})$ to reflections $\widetilde{t}_{i_1}$, \ldots, $\widetilde{t}_{i_{n - 1}}$ in $G(m, p, n)$ so that $g = \widetilde{t}_1 \cdots \widetilde{t}_{\ell}$.  To complete the counting argument, it remains to show that this unique choice is compatible with $f_1$, i.e., that $\pi_{m/p}(\widetilde{t}_{i_j}) = t_{i_j}$ for $j = 1, \ldots, n - 1$.  

Let $\widetilde{f} = (\widetilde{t}_1, \ldots,  \widetilde{t}_{\ell})$.  By construction, $\pi_{m/p}(\widetilde{f})$ and $f^*$ are both $G(p, p, n)$-factorizations of $\pi_{m/p}(g)$, with corresponding factors sharing the same underlying permutation, that agree 
at all positions except possibly $i_1, \ldots, i_{n - 1}$.  But by Lemma~\ref{tree lemma}, there is a unique $G(p, p, n)$-factorization of $\pi_{m/p}(g)$ that agrees with $f^*$ except at the positions $i_1, \ldots, i_{n - 1}$.  Thus $\pi_{m/p}(\widetilde{f}) = f^*$, so that the constructed factorization $\widetilde{f}$ really does map to the pair $(f_1, f_2)$.  

In summary, so far we have shown that there is a map from full length-$\ell$ $G(m, p, n)$-reflection factorizations of $g$ to pairs $(f_1, f_2)$ such that, for some integer $k$, $f_1$ is a full length-$k$ $G(p, p, n)$-reflection factorization of $\pi_{m/p}(g)$ and $f_2$ is a full length-$(\ell - k)$ $G(m, p, 1)$-reflection factorization of $\zeta_m^{\wt(g)}$, and that each such pair $(f_1, f_2)$ has $\binom{\ell}{k} (m/p)^{k - (n - 1)} n^{\ell - k}$ preimages under this map.  Converting the preceding statement to a generating-function calculation, we have 
\begin{align*}
[z^\ell]& \FFFtr_{m, p, n}(g; z) =\frac{1}{\ell!} \cdot \# \left\{
    \begin{array}{cc}
    \text{length-$\ell$ $G(m, p, n)$-refn.\ facts.} \\ 
    \text{of $g$ that generate $G(m, p, n)$}
    \end{array}\right\} \\
& = \sum_{k} \frac{(m / p)^{k - (n - 1)} n^{\ell - k} }{k! (\ell - k)! }
\cdot \#\left\{\!
    \begin{array}{cc}
    \text{length-$k$ } \\ 
    \text{$G(p, p, n)$-refn.} \\ 
    \text{facts.\ of $\pi_{m/p}(g)$ that} \\
    \text{generate $G(p, p, n)$}
    \end{array}\!\right\} 
\cdot \#\left\{\!
    \begin{array}{cc}
    \text{length-$(\ell - k)$ } \\ 
    \text{$G(m, p, 1)$-refn.} \\
    \text{facts.\ of $\zeta_m^{\wt(g)}$ that} \\ 
    \text{generate $G(m, p, 1)$}
    \end{array}\!\right\} \\
& = (p / m)^{n - 1} \cdot 
\sum_{k} (m / p)^{k}  n^{\ell - k} \cdot [z^k]
\FFFtr_{p, p, n}( \pi_{m / p}(g); z) \cdot [z^{\ell - k}]\FFFtr_{m, p, 1}(\zeta_m^{\wt(g)}; z) \\
& = (p / m)^{n - 1} \cdot  [z^\ell] \left(\FFFtr_{p, p, n}( \pi_{m / p}(g); (m/p)\cdot z) \cdot \FFFtr_{m, p, 1}(\zeta_m^{\wt(g)}; n \cdot z)\right),
\end{align*}
and the proposition follows immediately.
\end{proof}

\subsection{From $G(p,p,n)$ to $\Symm_n$}

Next, we show how to express the generating function for full factorizations of an element $g \in G(p, p, n)$ in terms of the series of $\Symm_n$-full factorizations of its underlying permutation $\pi_{p/1}(g)$.

\begin{proposition}
\label{thm:gf for arbitrary elements in G(m, m, n)}
Fix an element $g \in G(p, p, n)$ with $k$ cycles, of colors $a_1, \ldots, a_k$, and let $d = \gcd(a_1, \ldots, a_k, p)$.  Then

\[
\FFFtr_{p,p,n}(g;z) = \frac{1}{p^{n - 1}}\sum_{r \mid d} \left( \mu(r)\cdot r^{n + k - 2}\cdot \FFFtr_{\Symm_n}\big(\pi_{p/1}(g); (p/r)\cdot z\big)\right),
\]
where $\mu$ is the number-theoretic M\"obius function.
\end{proposition}
(Since any two integer representatives of $a_i$ differ by a multiple of $p$, the number $d$ is well defined as an integer; it could alternatively be defined as the unique positive integer such that $\{a_1, \ldots, a_k\}$ generates the cyclic subgroup $d\ZZ/p\ZZ$ of $\ZZ/p\ZZ$.)
\begin{proof}
By hypothesis, the colors of all cycles of $g$ are multiples of $d$.  Therefore, recalling the characterization of conjugacy classes in $G(p, 1, n)$ from Section~\ref{sec:Background}, we have that $g$ is conjugate by an element of $G(p, 1, n)$ to an element of $G(p/d, p/d, n)$.  (Here, and in the rest of this proof, we write $G(p/r, p/r, n)$ for the particular subgroup of $G(p, p, n)$ defined in \eqref{def:G(m, p, n)}, rather than the isomorphism class of such groups.) Since $W = G(p, p, n)$ is a normal subgroup of $G(p, 1, n)$, this conjugation extends to a bijection between $W$-full reflection factorizations, and so for convenience, we replace $g$ with its conjugate in $G(p/d, p/d, n)$.

Applying $\pi_{p/1}$ to any reflection factorization of $g$ that is full with respect to $W$ produces a connected $\Symm_n$-factorization of $\pi_{p/1}(g)$.  The main idea of the proof is to count preimages of each such $\Symm_n$-factorization.   The remainder of the argument has three main components: first we enumerate the reflection subgroups of each type that may be generated by such a preimage; then we show that the number of preimages which generate a subgroup of a given type only depends on its isomorphism type; finally we count the preimages according to the type of subgroup they generate. The final answer then follows from an inclusion-exclusion argument.

Fix a connected transposition factorization $f = (t_1, \ldots, t_\ell)$ of $\pi_{p/1}(g)$.  Since $f$ is connected, every subgroup of $W = G(p, p, n)$ that is generated by a lift of $f$ is conjugate in $G(p, 1, n)$ to $G(p/r, p/r, n)$ for some $r \mid p$.  Furthermore, since the cycle colors of $g$ generate $d\ZZ/p\ZZ$, when we restrict to lifts of $f$ that are factorizations of $g$, we have by Lemma~\ref{generating sets lemma} that each one generates a subgroup conjugate to $G(p/r, p/r, n)$ for some $r \mid d$ (not just $r \mid p$).  Moreover, as observed in Section~\ref{sec:Background}, it is enough to allow conjugation only by \emph{diagonal} elements of $G(p, 1, n)$.

We next consider how many distinct $G(p, 1, n)$-conjugates of $H = G(p/r, p/r, n)$ in $W = G(p, p, n)$ contain the element $g = [u; d \cdot a]$.  Let $\delta = [\id; (d_1, \ldots, d_n)] \in G(p, 1, 1)^n \subset G(p, 1, n)$.  By \eqref{eq:wreath product}, we have $\delta^{-1} [v; (b_1, \ldots, b_n)] \delta = [v; (b_1 + d_1 - d_{v(1)}, \ldots, b_n + d_n - d_{v(n)})]$ for any $[v; b] \in W$.  Thus if $g = \delta^{-1} g' \delta$ then $g' = [u; b]$ for some $n$-tuple $b = (b_1, \ldots, b_n) \in (\ZZ/p\ZZ)^n$ of colors.  Consequently, $g \in \delta^{-1} H \delta$ if and only if 
\[
[u; d \cdot a] = [u; (r b_1 + d_1 - d_{u(1)}, \ldots, r b_n + d_n - d_{u(n)})]
\]
for some $b$.  Since $r \mid d$, such a tuple $b$ exists if and only if $r \mid d_i - d_{u(i)}$ for $i = 1, \ldots, n$.  There are $p^k \cdot (p/r)^{n - k}$ choices $\delta$ such that these equations hold: the color of one element from each of the $k$ cycles of $u$ may be chosen arbitrarily, and the remaining colors in the cycle may be chosen to be any of the $p/r$ colors that differ from the first choice by a multiple of $r$.  Among these choices for $\delta$, there are $p \cdot (p/r)^{n - 1}$ that normalize $H$ (every two entries of $\delta$ must differ by a multiple of $r$), so by the orbit-stabilizer theorem there are $r^{k - 1}$ distinct copies of $H$ that contain $g$.

Furthermore, suppose that $\delta^{-1} H \delta$ is an isomorphic copy of $H$ containing $g$, so that 
\[
g = \delta^{-1} [u; r \cdot b] \delta = [u; (r b_1 + d_1 - d_{u(1)}, \ldots, r b_n + d_n - d_{u(n)})].
\]
Choose a new diagonal element $\delta' = [\id; (d'_1, \ldots, d'_n)]$ as follows: select one entry $j$ in each cycle of $u$ and set $d'_j = 0$, and determine the other values of $d'_i$ by the relation $d'_i - d'_{u(i)} = d_i - d_{u(i)}$ for $i = 1, \ldots, n$.  By construction, we have (first) that $(\delta')^{-1} [u; r \cdot b] \delta' = g$, and (second) that $d_i \in r\ZZ/p\ZZ$ for each $i$ and so $\delta' \in H$.  Consequently $\delta^{-1} \delta'$ commutes with $g$ and $(\delta^{-1} \delta') H (\delta^{-1} \delta')^{-1} = \delta^{-1} H \delta$.  Thus,  conjugation by $\delta^{-1} \delta'$ extends to a bijection between lifts of $f$ that are $H$-reflection factorizations of $g$ and lifts of $f$ that are $\delta^{-1} H \delta$-reflection factorizations of $g$ -- or in other words, all the $r^{k - 1}$ distinct copies of $H$ that contain $g$ also contain the same number of factorizations of $g$ that are lifts of $f$.

Now let us count lifts of $f$ according to what subgroup they generate.  Let $a_{p/r}$ be the number of lifts of $f$ that factor $g$ and generate the group $G(p/r, p/r, n)$, and let $b_{p/r}$ be the (generally larger) number of lifts of $f$ that factor $g$ and generate any subgroup of $G(p/r, p/r, n)$.  We next establish a relationship between the $a$s and the $b$s.  It follows from the arguments of the two last paragraphs that every lift of $f$ that factors $g$ and generates a subgroup of $G(p/r, p/r, n)$ in particular generates, for some integer $r'$ such that $r \mid r' \mid d$, one of the $(r'/r)^{k - 1}$ distinct subgroups of $G(p/r, p/r, n)$ that are isomorphic to $G(p/r', p/r', n)$ and contain $g$; and that the number of factorizations that generate each of these subgroups is $a_{p/r'}$.  Therefore for any $r \mid d$ we have
\[
b_{p/r} = \sum_{r' \colon r \mid r' \mid d} (r'/r)^{k - 1} \cdot a_{p/r'},
\]
or equivalently
\[
%\frac{b_{p/r}}{(p/r)^{k - 1}} = \sum_{r' \colon r \mid r' \mid d} \frac{a_{p/r'}}{(p/r')^{k-1}}.
r^{k - 1} \cdot b_{p/r} = \sum_{r' \colon r \mid r' \mid d} (r')^{k - 1} \cdot a_{p/r'}.
\]
By M\"obius inversion, it follows that the number $a_p = a_{p/1}$ of lifts of $f$ that factor $g$ and generate the full group $W = G(p, p, n)$ is 
\begin{equation}
\label{eq:post-mobius}
%a_{p} = p^{k - 1} \cdot \sum_{  r \mid d } \mu(r) \cdot \frac{b_{p/r}}{(p/r)^{k - 1}}.
a_{p} = \sum_{  r \mid d } \mu(r) \cdot r^{k - 1} \cdot b_{p/r}.
\end{equation}
Next, we compute the number $b_{p/r}$.

Since $f$ is connected, Lemma~\ref{tree lemma} promises a special subsequence $t_{i_1}, \ldots, t_{i_{n - 1}}$ with the following property: each of the $p^{\ell - (n - 1)}$ ways of lifting the $t_j$ for $j \neq i_1, \ldots, i_{n - 1}$ into $W$ determines a unique lift of $f$ to a $W$-factorization of $g$.  
Moreover, if each of the $\ell - (n - 1)$ non-special factors is lifted to a reflection in $G(p/r, p/r, n)$, then they and the product $g$ all live inside $G(p/r, 1, n)$; in this case, Lemma~\ref{tree lemma} promises that the remaining special factors will also lift to transposition-like reflections inside $G(p/r, 1, n)$.  Thus, $b_{p/r} = (p/r)^{\ell - (n - 1)}$ of the lifts generate a subgroup of $G(p/r, p/r, n)$.  Substituting this into \eqref{eq:post-mobius}, we conclude that the number $a_p = a_{p/1}$ of lifts of $f$ that factor $g$ and generate the full group $W = G(p, p, n)$ is 
\[
a_{p} = p^{k - 1} \cdot \sum_{  r \mid d } \mu(r) \cdot (p/r)^{\ell - n - k + 2}.
\]
Now taking into account all choices of $f$, we have
\[
[z^\ell] \FFFtr_{p, p, n}(g; z) = p^{k - 1} \cdot \sum_{  r \mid d } \mu(r) \cdot (p/r)^{\ell - n - k + 2} \cdot [z^\ell] \FFFtr_{\Symm_n}(g; z).
\]
The desired result follows immediately.
\end{proof}

\subsection{Completing the proof of Theorem~\ref{thm:main}}

Let $g$ be an arbitrary element of $W = G(m, p, n)$, and suppose that $g$ has $k$ cycles, of colors $a_1, \ldots, a_k \in \ZZ/m\ZZ$.  By Proposition~\ref{thm:gf for arbitrary elements in G(m, p, n)}, we have
\begin{equation}
    \label{eq:step 1}
    \FFFtr_{m, p, n}(g; z) = \dfrac{1}{(m/p)^{n-1}} \cdot \FFFtr_{p, p, n}( \pi_{m / p}(g); (m/p)\cdot z) \cdot \FFFtr_{m, p, 1}(\zeta_m^{\wt(g)}; n \cdot z).
\end{equation}
Viewed as an element of $G(m, p, n)$, $\pi_{m/p}(g)$ has cycles of colors $\frac{m}{p} a_i \in \ZZ/m\ZZ$; therefore, when viewed as an element of $G(p, p, n)$, its cycles have colors $a_i \in \ZZ/p\ZZ$.  Setting $d = \gcd(a_1, \ldots, a_k, p)$, we have by Proposition~\ref{thm:gf for arbitrary elements in G(m, m, n)} that
\begin{align*}
\FFFtr_{p, p, n}( \pi_{m / p}(g); (m/p)\cdot z) & = \frac{1}{p^{n - 1}}\sum_{r \mid d} \left( \mu(r)\cdot r^{n + k - 2}\cdot \FFFtr_{\Symm_n}\big(\pi_{p/1}(\pi_{m/p}(g)); (p/r)\cdot (m/p) \cdot z\big)\right) \\
    & = \frac{1}{p^{n - 1}}\sum_{r \mid d} \left( \mu(r)\cdot r^{n + k - 2}\cdot \FFFtr_{\Symm_n}\big(\pi_{m/1}(g); (m/r) \cdot z\big)\right).
\end{align*}
Plugging this into \eqref{eq:step 1} immediately gives the result. \hfill $\square$

\section{Recovering leading terms}
\label{sec:leading terms}

In this section we extract the leading term of the generating series $\FFFtr_{m,p,n}(g;z)$ in Theorem~\ref{thm:main} to obtain the number of \emph{minimum-length} full factorizations of an arbitrary element in $G(m, p, n)$. The answer will involve  the Euler totient function $\varphi(m)$,  the \defn{Jordan totient function}
\begin{equation} \label{eq:prod psi and phi}
J_2(m) := \sum_{d\mid m} \mu(m/d) \cdot d^2,
\end{equation}
which counts elements of order $m$ in the group $(\ZZ/m\ZZ)^2$ \cite[\href{https://oeis.org/A007434}{A007434}]{oeis}, and Hurwitz numbers of the symmetric group of genus $0$ (given by Theorem~\ref{thm:S_n genus 0}) and genus $1$. The latter also have an explicit formula.

\begin{theorem}[{Goulden--Jackson \cite{GJ99}; Vakil \cite{V01}}]
\label{GJVGJV}
  For $\lambda=(\lambda_1,\ldots,\lambda_k)$, the number of genus-$1$ transitive transposition factorizations in $\Symm_n$ of a permutation of cycle type $\lambda$ is
  \[
H_1(\lambda) = \frac{1}{24} (n+k)! \left(\prod_{i=1}^k
  \frac{\lambda_i^{\lambda_i}}{(\lambda_i-1)!}\right) \left( n^k -
  n^{k-1}  - \sum_{i=2}^k (i-2)! \cdot e_i(\lambda)\cdot  n^{k-i}\right),
\]
where $e_i$ denotes the $i$th elementary symmetric function.
\end{theorem}

In order to extract the leading coefficients of $\FFFtr_{m,p,n}(g;z)$, we need first to determine the full reflection length $\ltr(g)$ for $g$ in $W = G(m,p,n)$. We do this by calculating from Theorem~\ref{thm:main} the degree of the leading term in the generating function $\FFFtr_{m,p,n}(g;z)$.  Although the relation between generating functions in Theorem~\ref{thm:main} is stated uniformly for all $m, p, n$, it is most convenient to formulate our corollary separately for $G(m, m, n)$ and for $G(m, p, n)$ with $p < m$. This is because of the appearance of the cyclic group $p\ZZ/m\ZZ$ in Theorem~\ref{thm:main} -- when $p = m$, this group is trivial and does not contribute to the (full) reflection length (see Remark~\ref{rem: main theorem special cases}).

\begin{corollary}
\label{cor:full length}
Let $W = G(m, p, n)$.
For an element $g\in W$ with $k$ cycles, of colors $a_1, \ldots, a_k$, let $d = \gcd(a_1, \ldots, a_k, p)$  and $a=\gcd(\col(g),m)/p$. If $m=p$, we have 
    \[
\ltr(g) = \begin{cases}
    n + k - 2, & \text{ if } d = 1 \\
    n + k , & \text{ if } d \neq 1,
    \end{cases}
    \]
while if $m\neq p$, we have
    \[
    \ltr(g) = \begin{cases}
    n + k - 1, & \text{ if } a=1 \text{ and } d = 1 \\
    n + k, & \text{ if } a\neq 1 \text{ and }  d = 1\\
    n + k + 1, & \text{ if } a=1 \text{ and }  d \neq 1\\
    n + k + 2, & \text{ if } a\neq 1 \text{ and }  d \neq 1.
    \end{cases}
    \]

\end{corollary}

\begin{proof}
We need only compute the degree (in $z$) of the lowest-order term of the generating function $\FFFtr_{m,p,n}(g; z)$. Looking at the right side of the equation in Theorem~\ref{thm:main}, we make the following observations. 
\begin{enumerate}[(1)]
    \item When $m=p$, the group $G(m,m,1)$ is the trivial group and the generating function $\FFFtr_{m,m,1}(\zeta_m^{\wt(g)};n\cdot z)$ equals $1$ as we explain in Remark~\ref{rem: main theorem special cases}, contributing the factor $z^0$ for the degree of the lowest-order term of $\FFFtr_{m,p,n}(g;z)$.
    \item When $m\neq p$, the group $G(m,p,1)$ is the cyclic group of order $m/p$ and the generating function $\FFFtr_{m,p,1}(\zeta_m^{\wt(g)};n\cdot z)$ will either contribute a factor of $z^1$ (if $\zeta_m^{\wt(g)}$ generates $G(m,p,1)$) or $z^2$ (if not) to the degree of the lowest-order monomial.  The condition $\zeta_m^{\wt(g)}$ generates $G(m,p,1)$ is equivalent to $\gcd(\wt(g), m) = p$, i.e., to $a = 1$.
    \item When $d=1$, the sum factor has a unique term $\FFFtr_{\Symm_n}(\pi_{m/1}(g);m\cdot z)$.  By Theorem~\ref{thm:S_n genus 0}, since $\pi_{m/1}(g)$ has $k$ cycles, this generating function will contribute the factor $z^{n+k-2}$. 
    \item When $d\neq 1$, the coefficient of $z^{n+k-2}$ in $\FFFtr_{\Symm_n}(\pi_{m/1}(g);(m/r)\cdot z)$ is a multiple of 
    \[
    \sum_{r \colon r \mid d} \mu(r)\cdot r^{n + k - 2} \cdot (m/r)^{n+k-2}=m^{n+k-2}\cdot\sum_{r \colon r\mid d}\mu(r)=0,
    \] leaving $z^{n+k}$ as the contribution to the lowest-order monomial. Indeed, its coefficient 
    \[
    \sum_{r \colon r \mid d} \mu(r)\cdot r^{n + k - 2} \cdot (m/r)^{n+k}=\dfrac{m^{n+k}}{d^2}\sum_{b \colon b\mid d}\mu(d/b)\cdot b^2
    \]
    equals  $\dfrac{m^{n+k}}{d^2} J_2(d)$ by \eqref{eq:prod psi and phi} and thus is nonzero. 
\end{enumerate}
The statement of the corollary is immediate after the previous points.
\end{proof}

\begin{remark}
It is interesting to observe that the formulas for full reflection length in Corollary~\ref{cor:full length} are efficiently computable (indeed, the computation is completely straightforward).  By contrast, although an explicit combinatorial formula exists for reflection length in $G(m, p, n)$ \cite[Thm.~4.4]{Shi2007}, it is computationally intractable in general -- see \cite[Rem.~2.4]{LW}.
\end{remark}

\begin{theorem}
\label{thm:main lead coeff}
For an element $g\in G(m,p,n)$ with $k$ cycles, of colors $a_1, \ldots, a_k$, let $d = \gcd(a_1, \ldots, a_k, p)$ and $a=\gcd(\col(g),m)/p$. If $m=p$, we have 
    \[
\Ftr_{m,m,n}(g) = \begin{cases}
    m^{k-1} \cdot H_0(\lambda), & \text{ if } d = 1 \\
    m^{k+1} \cdot \frac{J_2(d)}{d^2} \cdot  H_1(\lambda), & \text{ if } d \neq 1,
    \end{cases}
    \]
while if $m\neq p$, we have
    \[
    \Ftr_{m,p,n}(g) = \begin{cases}
    n(n+k-1)\cdot m^{k-1}  \cdot H_0(\lambda), & \text{ if } a=1 \text{ and } d = 1 \\[6pt]
    \frac{n^2 (n+k)(n+k-1)m^{k} }{2} \cdot \frac{\varphi(a)}{pa} \cdot H_0(\lambda), & \text{ if } a\neq 1 \text{ and }  d = 1\\[6pt]
    n(n+k+1)m^{k+1} \cdot \frac{J_2(d)}{d^2} \cdot H_1(\lambda), & \text{ if } a=1 \text{ and }  d \neq 1\\[6pt]
    \frac{n^2(n+k+2)(n+k+1)m^{k+2}}{2} \cdot  \frac{\varphi(a)}{pa} \cdot \frac{J_2(d)}{d^2}\cdot  H_1(\lambda), & \text{ if } a \neq 1 \text{ and }  d \neq 1,
    \end{cases}
    \]
where  $\lambda$ is the cycle type of the underlying permutation $\pi_{m/1}(g)$.
\end{theorem}

\begin{proof}
Since $g$ has $k$ cycles, the full reflection length of the permutation $\pi_{m/1}(g)$ is $n+k-2$. 
We consider the same cases as of Corollary~\ref{cor:full length}. 

If $m=p$, then in all cases $\FFFtr_{m,m,1}\big(\zeta_m^{\wt(g)};n\cdot z\big)=1$.  If $d=1$ then $\ltr(g)=n+k-2$ by Corollary~\ref{cor:full length}.  Thus by Theorem~\ref{thm:main} we have 
\begin{align*}
\Ftr_{m,m,n}(g) & = \frac{(n+k-2)!}{m^{n-1}} \cdot [z^{n+k-2}]\, \FFFtr_{\Symm_n}\big(\pi_{m/1}(g);m\cdot z\big) \\
& = \frac{(n+k-2)!}{m^{n-1}} \cdot \frac{m^{n+k-2} H_0(\lambda)}{(n+k-2)!} \\
&= m^{k-1} \cdot H_0(\lambda).
\end{align*}
If instead $d\neq 1$, then $\ltr(g)=n+k$  by Corollary~\ref{cor:full length}.  Thus by Theorem~\ref{thm:main} we have 
\begin{align*}
\Ftr_{m,m,n}(g) &= \frac{(n+k)!}{m^{n-1}} \cdot [z^{n+k}]\, \sum_{r: r\mid d} \left(\mu(r) r^{n+k-2} \FFFtr_{\Symm_n}\big(\pi_{m/1}(g);(m/r)\cdot z\big)\right)\\
&=\frac{1}{m^{n-1}} \cdot H_1(\lambda) \cdot \sum_{r: r\mid d} \mu(r)  \cdot r^{n+k-2} \cdot (m/r)^{n+k}.
\end{align*}
By \eqref{eq:prod psi and phi}, this simplifies to
\[
\frac{1}{m^{n-1}} \cdot H_1(\lambda) \cdot \frac{m^{n+k}}{d^2} \cdot J_2(d) = \frac{m^{k+1}}{d^2}\cdot J_2(d) \cdot  H_1(\lambda),
\]
as claimed.

Now suppose $m\neq p$. Observe that $ \frac{m}{\gcd(\col(g), m)} = \frac{m}{pa}$ is precisely the order of $\col(g)$ in the cyclic group $\ZZ/m\ZZ$ (or any subgroup thereof that contains it).  In particular, $\zeta_m^{\col(g)}$ generates $G(m, p, 1) \cong \ZZ/(m/p)\ZZ$ if and only if $a = 1$.

If $d = 1$, then, as in the case $p = m$, the summation in Theorem~\ref{thm:main} consists of a single term, whose lowest-degree term is $H_0(\lambda) \cdot \frac{m^{n + k - 2} z^{n + k - 2}}{(n + k - 2)!}$.  If $a = 1$, so that $\zeta_m^{\col(g)}$ generates $G(m, p, 1)$, then the lowest-degree term of $\FFFtr_{m,p,1}(\zeta_m^{\wt(g)};n\cdot z)$ is $n\cdot z$.  It follows that the term of degree $\ltr(g) = n + k - 1$ in this case is 
\[
\Ftr_{m,p,n}(g) = \frac{(n+k-1)!}{m^{n-1}} \cdot H_0(\lambda) \cdot \frac{m^{n + k - 2} z^{n + k - 2}}{(n + k - 2)!} \cdot n = n(n+k-1)\cdot m^{k-1}  \cdot H_0(\lambda),
\]
as claimed.

Still considering the case $p \neq m$ and $d = 1$, let us suppose instead that $a \neq 1$.  In this case, the contribution from $\FFFtr_{m,p,1}(\zeta_m^{\wt(g)};n\cdot z)$ is $c \cdot n^2 \cdot z^2/2$ where $c$ is the number of full factorizations of length $2$ of $\zeta_m^{\wt(g)}$ in the cyclic group $G(m, p, 1)$.  More generally, let $a(R, N)$ be the number of length-$2$ (not necessarily reflection) factorizations of an element of order $N/R$ in the cyclic group of order $N$ that do not lie in a proper subgroup. The number of all length-$2$ factorizations of such an element is $b(R, N) = N$, and consequently 
\[
N = \sum_{r \colon r \mid R} a(R/r, N/r)
\]
for all $N$ such that $R \mid N$.  By M\"obius inversion, it follows that
%\begin{align*}
\[
a(R, N) = \sum_{r : r\mid R} \mu(r) \cdot b(R/r, N/r) = \frac{N}{R} \sum_{r:r\mid R} \mu(r) \cdot R/r  = \frac{N}{R} \varphi(R).
\]
%\end{align*}
Now specializing to our particular case $N = \frac{m}{p}$ and $R = a$, since the element of order $N/R$ is not a generator, the factorizations counted by $a(R,N)$ are actually full reflection factorizations. Thus we have $c= \frac{m}{pa} \varphi(a)$ and we obtain
\begin{align*}
\Ftr_{m,p,n}(g) 
& = \frac{n^2 \cdot (n+k)! }{2\cdot m^{n-1}} \cdot \frac{m}{pa} \cdot \varphi(a)\cdot  \frac{m^{n + k + 2} \cdot H_0(\lambda)}{(n + k - 2)!} \\
&= \frac{n^2 (n+k)(n+k-1)m^{k} }{2pa} \cdot \varphi(a) \cdot H_0(\lambda),
\end{align*}
as claimed.

On the other hand, if $d \neq 1$, then, again as in the case $p = m$, the lowest-degree term from the summation factor is 
\[
 [z^{n+k}] \, \sum_{r, r\mid d} \left(\mu(r) r^{n+k-2} \FFFtr_{\Symm_n}\big(\pi_{m/1}(g);(m/r)\cdot z\big)\right) = \frac{m^{n + k} \cdot J_2(d) \cdot H_1(\lambda)}{(n + k)! \cdot d^2} .
\]
Consequently, when $a = 1$ we have
\begin{align*}
    \Ftr_{m, p, n}(g) & = \frac{(n + k + 1)!}{m^{n - 1}} \cdot n \cdot \frac{m^{n + k} \cdot J_2(d) \cdot H_1(\lambda)}{(n + k)! \cdot d^2} \\
    & = \frac{n(n + k + 1) m^{k + 1}}{d^2} \cdot J_2(d) \cdot H_1(\lambda),
\end{align*}
while when $a \neq 1$ we instead have
\begin{align*}
\Ftr_{m,p,n}(g) & = \frac{n^2 \cdot (n+k+2)!}{2\cdot m^{n-1}} \cdot \frac{m}{pa} \cdot \varphi(a) \cdot \frac{m^{n + k} \cdot J_2(d) \cdot H_1(\lambda)}{(n + k)! \cdot d^2}\\
&= \frac{n^2(n + k + 1)(n + k + 2) m^{k + 2}}{2d^2pa} \cdot \varphi(a) \cdot J_2(d) \cdot H_1(\lambda),
\end{align*}
as claimed.
\end{proof}

\begin{remark}
If $W$ is a complex reflection group of exceptional type, we can recover the number $\Ftr_W(g)$ of minimum-length full reflection factorizations of an element $g$ by computing the series $\FFFtr_{W}(g; z)$ as in Section~\ref{Sec: Frob lemma} and extracting the lowest-order term of the series.  When we express the generating function as a Laurent polynomial in $X = e^z$ as in Proposition~\ref{Prop: structural result FFFtr}, this gives
\[
\Ftr_W(g) = \frac{1}{\# W} \cdot \Phi_W(g; 1) \cdot \ltr(g)!.
\]
For example, continuing where Example~\ref{Example: FFFtr(H3)} left off, we have that $\Phi_{H_3}(\id; 1) = 28800$ and so the number of minimum-length full reflection factorizations of the identity in $H_3$ is given by
\begin{align*}
%\label{eq: FFFtr(H_3)}
\Ftr_{H_3}(\id) & = \dfrac{1}{\#H_3} \cdot \Phi_{H_3}(\id; 1)\cdot \ltr[H_3](\id)! \\
                & = \frac{1}{120} \cdot 28800 \cdot 6! \\
                & = 172800.
\end{align*}
\end{remark}

\section*{Acknowledgements}

We thank David Jackson and Jiayuan Wang  for helpful comments and suggestions.
This work was facilitated by computer experiments using Sage \cite{sagemath}, its algebraic combinatorics features developed by the Sage-Combinat community \cite{Sage-Combinat}, and CHEVIE \cite{chevie}. 

The second-named author was partially supported by an ORAU Powe award and by a grant from the Simons Foundation (634530). The third-named author  was partially supported by NSF grant DMS-1855536.  

An extended abstract of this work appeared as \cite{DLM-extended-abstract}.

\printbibliography

\newgeometry{margin=.5in}

\appendix

\section{Polynomials and their roots}
\label{appendix}

Below, we give the polynomials $\Phi_W(\id; X)$ (as in Proposition~\ref{Prop: structural result FFFtr}) for exceptional real reflection groups $W$.  The polynomials $\Phi_W(\id; X)$ for all exceptional complex reflection groups $W$ are attached to this arXiv submission as a supplementary file.  For all exceptional complex groups, we give figures showing the location of their roots in the complex plane.

\subsection[Polynomial for the exceptional real reflection groups]{The polynomials $\Phi_W(\id; X)$ for exceptional real reflection groups}

There are six exceptional real reflection groups, of types $H_3$, $H_4$, $F_4$, $E_6$, $E_7$, and $E_8$ (in the Shephard--Todd classification, respectively $G_{23}$, $G_{30}$, $G_{28}$, $G_{35}$, $G_{36}$, and $G_{37}$).  We also include the dihedral group of order $12$, $I_2(6)=G_2=G(6,6,2)$, for which the coefficients of $\Phi_{G_2}(\id;X)$ are \emph{not} unimodal.

\begin{dmath*}
\Phi_{G_2}(\id; X)=X^8 + 4X^7 + 10X^6 + 16X^5 + 10X^4 + 16X^3 + 10X^2 + 4X + 1
\end{dmath*}

\begin{dmath*}
\Phi_{H_3}(\id; X) =  
X^{24} + 6 X^{23} + 21 X^{22} + 56 X^{21} + 126 X^{20} + 252 X^{19} + 462 X^{18} + 792 X^{17} + 1287 X^{16} + 2002 X^{15} + 2949 X^{14} + 4044 X^{13} + 4804 X^{12} + 4044 X^{11} + 2949 X^{10} + 2002 X^{9} + 1287 X^{8} + 792 X^{7} + 462 X^{6} + 252 X^{5} + 126 X^{4} + 56 X^{3} + 21 X^{2} + 6 X + 1
\end{dmath*}

\begin{dmath*}
\Phi_{H_4}(\id; X) =  
X^{112} + 8 X^{111} + 36 X^{110} + 120 X^{109} + 330 X^{108} + 792 X^{107} + 1716 X^{106} + 3432 X^{105} + 6435 X^{104} + 11440 X^{103} + 19448 X^{102} + 31824 X^{101} + 50388 X^{100} + 77520 X^{99} + 116280 X^{98} + 170544 X^{97} + 245157 X^{96} + 346104 X^{95} + 480700 X^{94} + 657800 X^{93} + 888030 X^{92} + 1184040 X^{91} + 1560780 X^{90} + 2035800 X^{89} + 2629575 X^{88} + 3365856 X^{87} + 4272048 X^{86} + 5379616 X^{85} + 6724520 X^{84} + 8347680 X^{83} + 10295504 X^{82} + 12620512 X^{81} + 15382089 X^{80} + 18647400 X^{79} + 22492500 X^{78} + 27003672 X^{77} + 32279026 X^{76} + 38430392 X^{75} + 45585540 X^{74} + 53890760 X^{73} + 63513997 X^{72} + 74648736 X^{71} + 87518832 X^{70} + 102384480 X^{69} + 119545920 X^{68} + 139341984 X^{67} + 162136992 X^{66} + 188289504 X^{65} + 218095185 X^{64} + 251696040 X^{63} + 288937188 X^{62} + 329152344 X^{61} + 370859178 X^{60} + 411345720 X^{59} + 446433380 X^{58} + 470504008 X^{57} + 478642194 X^{56} + 470504008 X^{55} + 446433380 X^{54} + 411345720 X^{53} + 370859178 X^{52} + 329152344 X^{51} + 288937188 X^{50} + 251696040 X^{49} + 218095185 X^{48} + 188289504 X^{47} + 162136992 X^{46} + 139341984 X^{45} + 119545920 X^{44} + 102384480 X^{43} + 87518832 X^{42} + 74648736 X^{41} + 63513997 X^{40} + 53890760 X^{39} + 45585540 X^{38} + 38430392 X^{37} + 32279026 X^{36} + 27003672 X^{35} + 22492500 X^{34} + 18647400 X^{33} + 15382089 X^{32} + 12620512 X^{31} + 10295504 X^{30} + 8347680 X^{29} + 6724520 X^{28} + 5379616 X^{27} + 4272048 X^{26} + 3365856 X^{25} + 2629575 X^{24} + 2035800 X^{23} + 1560780 X^{22} + 1184040 X^{21} + 888030 X^{20} + 657800 X^{19} + 480700 X^{18} + 346104 X^{17} + 245157 X^{16} + 170544 X^{15} + 116280 X^{14} + 77520 X^{13} + 50388 X^{12} + 31824 X^{11} + 19448 X^{10} + 11440 X^{9} + 6435 X^{8} + 3432 X^{7} + 1716 X^{6} + 792 X^{5} + 330 X^{4} + 120 X^{3} + 36 X^{2} + 8 X + 1
\end{dmath*}

\begin{dmath*}
\Phi_{F_4}(\id; X) = X^{40} + 8 X^{39} + 36 X^{38} + 120 X^{37} + 330 X^{36} + 792 X^{35} + 1716 X^{34} + 3432 X^{33} + 6417 X^{32} + 11296 X^{31} + 18800 X^{30} + 29664 X^{29} + 44496 X^{28} + 63648 X^{27} + 87120 X^{26} + 114528 X^{25} + 144942 X^{24} + 176400 X^{23} + 204904 X^{22} + 222704 X^{21} + 225612 X^{20} + 222704 X^{19} + 204904 X^{18} + 176400 X^{17} + 144942 X^{16} + 114528 X^{15} + 87120 X^{14} + 63648 X^{13} + 44496 X^{12} + 29664 X^{11} + 18800 X^{10} + 11296 X^{9} + 6417 X^{8} + 3432 X^{7} + 1716 X^{6} + 792 X^{5} + 330 X^{4} + 120 X^{3} + 36 X^{2} + 8 X + 1
\end{dmath*}

\begin{dmath*}
\Phi_{E_6}(\id; X) =  
X^{60} + 12 X^{59} + 78 X^{58} + 364 X^{57} + 1365 X^{56} + 4368 X^{55} + 12376 X^{54} + 31824 X^{53} + 75582 X^{52} + 167960 X^{51} + 352716 X^{50} + 705432 X^{49} + 1352114 X^{48} + 2496576 X^{47} + 4460208 X^{46} + 7739264 X^{45} + 13086306 X^{44} + 21622680 X^{43} + 34986364 X^{42} + 55512408 X^{41} + 86428086 X^{40} + 132016816 X^{39} + 197656008 X^{38} + 289553520 X^{37} + 413889127 X^{36} + 574909668 X^{35} + 771474762 X^{34} + 991891204 X^{33} + 1208223291 X^{32} + 1374825408 X^{31} + 1439491744 X^{30} + 1374825408 X^{29} + 1208223291 X^{28} + 991891204 X^{27} + 771474762 X^{26} + 574909668 X^{25} + 413889127 X^{24} + 289553520 X^{23} + 197656008 X^{22} + 132016816 X^{21} + 86428086 X^{20} + 55512408 X^{19} + 34986364 X^{18} + 21622680 X^{17} + 13086306 X^{16} + 7739264 X^{15} + 4460208 X^{14} + 2496576 X^{13} + 1352114 X^{12} + 705432 X^{11} + 352716 X^{10} + 167960 X^{9} + 75582 X^{8} + 31824 X^{7} + 12376 X^{6} + 4368 X^{5} + 1365 X^{4} + 364 X^{3} + 78 X^{2} + 12 X + 1
\end{dmath*}

{\fontsize{5}{12}
\begin{dmath*}
\Phi_{E_7}(\id; X) =  
X^{112} + 14 X^{111} + 105 X^{110} + 560 X^{109} + 2380 X^{108} + 8568 X^{107} + 27132 X^{106} + 77520 X^{105} + 203490 X^{104} + 497420 X^{103} + 1144066 X^{102} + 2496144 X^{101} + 5200300 X^{100} + 10400600 X^{99} + 20058300 X^{98} + 37442160 X^{97} + 67863915 X^{96} + 119759850 X^{95} + 206253124 X^{94} + 347374286 X^{93} + 573171585 X^{92} + 928011200 X^{91} + 1476454420 X^{90} + 2311209432 X^{89} + 3563796768 X^{88} + 5418748776 X^{87} + 8132396454 X^{86} + 12057594892 X^{85} + 17676114371 X^{84} + 25640887902 X^{83} + 36830861774 X^{82} + 52421838406 X^{81} + 73977408477 X^{80} + 103564791120 X^{79} + 143901037014 X^{78} + 198535419396 X^{77} + 272073644850 X^{76} + 370448301048 X^{75} + 501237054234 X^{74} + 674024525076 X^{73} + 900794021190 X^{72} + 1196319168576 X^{71} + 1578499592916 X^{70} + 2068545055416 X^{69} + 2690855030430 X^{68} + 3472362845868 X^{67} + 4441013886057 X^{66} + 5622932352114 X^{65} + 7037726195703 X^{64} + 8691351504144 X^{63} + 10566155272554 X^{62} + 12608421831516 X^{61} + 14715363474420 X^{60} + 16726415799996 X^{59} + 18427842412154 X^{58} + 19582815012880 X^{57} + 19994199701232 X^{56} + 19582815012880 X^{55} + 18427842412154 X^{54} + 16726415799996 X^{53} + 14715363474420 X^{52} + 12608421831516 X^{51} + 10566155272554 X^{50} + 8691351504144 X^{49} + 7037726195703 X^{48} + 5622932352114 X^{47} + 4441013886057 X^{46} + 3472362845868 X^{45} + 2690855030430 X^{44} + 2068545055416 X^{43} + 1578499592916 X^{42} + 1196319168576 X^{41} + 900794021190 X^{40} + 674024525076 X^{39} + 501237054234 X^{38} + 370448301048 X^{37} + 272073644850 X^{36} + 198535419396 X^{35} + 143901037014 X^{34} + 103564791120 X^{33} + 73977408477 X^{32} + 52421838406 X^{31} + 36830861774 X^{30} + 25640887902 X^{29} + 17676114371 X^{28} + 12057594892 X^{27} + 8132396454 X^{26} + 5418748776 X^{25} + 3563796768 X^{24} + 2311209432 X^{23} + 1476454420 X^{22} + 928011200 X^{21} + 573171585 X^{20} + 347374286 X^{19} + 206253124 X^{18} + 119759850 X^{17} + 67863915 X^{16} + 37442160 X^{15} + 20058300 X^{14} + 10400600 X^{13} + 5200300 X^{12} + 2496144 X^{11} + 1144066 X^{10} + 497420 X^{9} + 203490 X^{8} + 77520 X^{7} + 27132 X^{6} + 8568 X^{5} + 2380 X^{4} + 560 X^{3} + 105 X^{2} + 14 X + 1
\end{dmath*}

\begin{dmath*}
\Phi_{E_8}(\id; X) =  
 X^{224} + 16 X^{223} + 136 X^{222} + 816 X^{221} + 3876 X^{220} + 15504 X^{219} + 54264 X^{218} + 170544 X^{217} + 490314 X^{216} + 1307504 X^{215} + 3268760 X^{214} + 7726160 X^{213} + 17383860 X^{212} + 37442160 X^{211} + 77558760 X^{210} + 155117520 X^{209} + 300540195 X^{208} + 565722720 X^{207} + 1037158320 X^{206} + 1855967520 X^{205} + 3247943160 X^{204} + 5567902560 X^{203} + 9364199760 X^{202} + 15471286560 X^{201} + 25140840660 X^{200} + 40225345056 X^{199} + 63432274896 X^{198} + 98672427616 X^{197} + 151532656696 X^{196} + 229911617056 X^{195} + 344867425648 X^{194} + 511738761568 X^{193} + 751616313253 X^{192} + 1093260131568 X^{191} + 1575580950648 X^{190} + 2250830567376 X^{189} + 3188678704316 X^{188} + 4481392321136 X^{187} + 6250379131016 X^{186} + 8654411335376 X^{185} + 11899909726430 X^{184} + 16253743972880 X^{183} + 22059094029480 X^{182} + 29755022456880 X^{181} + 39900530655180 X^{180} + 53204016714000 X^{179} + 70559222415000 X^{178} + 93088956174000 X^{177} + 122198112403600 X^{176} + 159637781803600 X^{175} + 207582568350376 X^{174} + 268723605356016 X^{173} + 346380204411636 X^{172} + 444633588497616 X^{171} + 568486767209176 X^{170} + 724055323460848 X^{169} + 918794715583618 X^{168} + 1161770678138928 X^{167} + 1463980454624568 X^{166} + 1838733945429648 X^{165} + 2302105439429540 X^{164} + 2873468456744240 X^{163} + 3576128407175240 X^{162} + 4438070311335440 X^{161} + 5492841789872915 X^{160} + 6780594951730976 X^{159} + 8349314754544016 X^{158} + 10256265912938336 X^{157} + 12569695526987816 X^{156} + 15370834309497376 X^{155} + 18756245597741680 X^{154} + 22840578197334880 X^{153} + 27759786428738180 X^{152} + 33674888371923680 X^{151} + 40776340990159280 X^{150} + 49289118211055776 X^{149} + 59478584664595816 X^{148} + 71657262954125536 X^{147} + 86192595152779216 X^{146} + 103515798425730976 X^{145} + 124131908566547077 X^{144} + 148631091445728432 X^{143} + 177701277635283672 X^{142} + 212142135277058832 X^{141} + 252880334399639052 X^{140} + 300985963949096496 X^{139} + 357689829578199336 X^{138} + 424401171085373456 X^{137} + 502725074539768286 X^{136} + 594478492077657296 X^{135} + 701703293576640136 X^{134} + 826674125714417776 X^{133} + 971898009271752796 X^{132} + 1140101530303514576 X^{131} + 1334200149310289336 X^{130} + 1557242560896681200 X^{129} + 1812321222812552775 X^{128} + 2102438251446038400 X^{127} + 2430314099127206400 X^{126} +
 2798125247794473600 X^{125} + 3207157343271082176 X^{124} + 3657362950903698816 X^{123} + 4146820218538481536 X^{122} + 4671102775055679616 X^{121} + 5222595588611999776 X^{120} + 5789830289248817536 X^{119} + 6356970527607856576 X^{118} + 6903653712807277696 X^{117} + 7405474472724543776 X^{116} + 7835434962493540736 X^{115} + 8166590077767088256 X^{114} + 8375728793722324096 X^{113} + 8447244468199766416 X^{112} + 8375728793722324096 X^{111} + 8166590077767088256 X^{110} + 7835434962493540736 X^{109} + 7405474472724543776 X^{108} + 6903653712807277696 X^{107} + 6356970527607856576 X^{106} + 5789830289248817536 X^{105} + 5222595588611999776 X^{104} + 4671102775055679616 X^{103} + 4146820218538481536 X^{102} + 3657362950903698816 X^{101} + 3207157343271082176 X^{100} + 2798125247794473600 X^{99} + 2430314099127206400 X^{98} + 2102438251446038400 X^{97} + 1812321222812552775 X^{96} + 1557242560896681200 X^{95} + 1334200149310289336 X^{94} + 1140101530303514576 X^{93} + 971898009271752796 X^{92} + 826674125714417776 X^{91} + 701703293576640136 X^{90} + 594478492077657296 X^{89} + 502725074539768286 X^{88} + 424401171085373456 X^{87} + 357689829578199336 X^{86} + 300985963949096496 X^{85} + 252880334399639052 X^{84} + 212142135277058832 X^{83} + 177701277635283672 X^{82} + 148631091445728432 X^{81} + 124131908566547077 X^{80} + 103515798425730976 X^{79} + 86192595152779216 X^{78} + 71657262954125536 X^{77} + 59478584664595816 X^{76} + 49289118211055776 X^{75} + 40776340990159280 X^{74} + 33674888371923680 X^{73} + 27759786428738180 X^{72} + 22840578197334880 X^{71} + 18756245597741680 X^{70} + 15370834309497376 X^{69} + 12569695526987816 X^{68} + 10256265912938336 X^{67} + 8349314754544016 X^{66} + 6780594951730976 X^{65} + 5492841789872915 X^{64} + 4438070311335440 X^{63} + 3576128407175240 X^{62} + 2873468456744240 X^{61} + 2302105439429540 X^{60} + 1838733945429648 X^{59} + 1463980454624568 X^{58} + 1161770678138928 X^{57} + 918794715583618 X^{56} + 724055323460848 X^{55} + 568486767209176 X^{54} + 444633588497616 X^{53} + 346380204411636 X^{52} + 268723605356016 X^{51} + 207582568350376 X^{50} + 159637781803600 X^{49} + 122198112403600 X^{48} + 93088956174000 X^{47} + 70559222415000 X^{46} + 53204016714000 X^{45} + 39900530655180 X^{44} + 29755022456880 X^{43} + 22059094029480 X^{42} + 16253743972880 X^{41} + 11899909726430 X^{40} + 8654411335376 X^{39} + 6250379131016 X^{38} + 4481392321136 X^{37} + 3188678704316 X^{36} + 2250830567376 X^{35} + 1575580950648 X^{34} + 1093260131568 X^{33} + 751616313253 X^{32} + 511738761568 X^{31} + 344867425648 X^{30} + 229911617056 X^{29} + 151532656696 X^{28} + 98672427616 X^{27} + 63432274896 X^{26} + 40225345056 X^{25} + 25140840660 X^{24} + 15471286560 X^{23} + 9364199760 X^{22} + 5567902560 X^{21} + 3247943160 X^{20} + 1855967520 X^{19} + 1037158320 X^{18} + 565722720 X^{17} + 300540195 X^{16} + 155117520 X^{15} + 77558760 X^{14} + 37442160 X^{13} + 17383860 X^{12} + 7726160 X^{11} + 3268760 X^{10} + 1307504 X^{9} + 490314 X^{8} + 170544 X^{7} + 54264 X^{6} + 15504 X^{5} + 3876 X^{4} + 816 X^{3} + 136 X^{2} + 16 X + 1
\end{dmath*}
}

\subsection[Polynomial roots for exceptional complex reflection groups]{Roots of the polynomials $\Phi_W(\id;X)$ for all exceptional complex reflection groups}

We give below the plot of roots of the polynomials $\Phi_W(\id; X)$ in the complex plane.  For the polynomials themselves, see the data file attached as a supplementary file to this arXiv submission.

\vspace{-3cm}

\subsection*{Rank 2}\ \newline

\vspace{-2cm}

\begin{figure}[H]
\begin{subfigure}[b]{.25\linewidth}
\centering \includegraphics[scale=0.25]{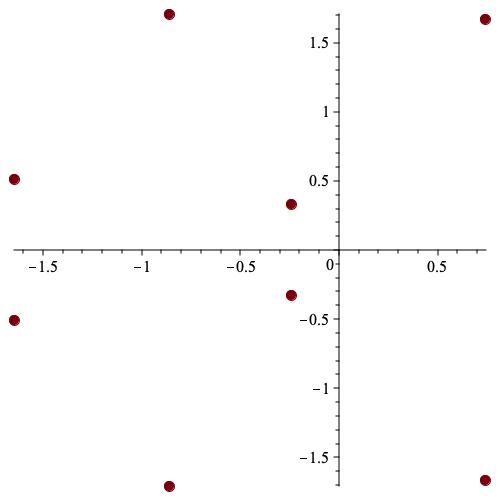}
\caption*{$G_4$}%\label{fig:1a}
\end{subfigure}\qquad\begin{subfigure}[b]{.25\linewidth}
\centering \includegraphics[scale=0.25]{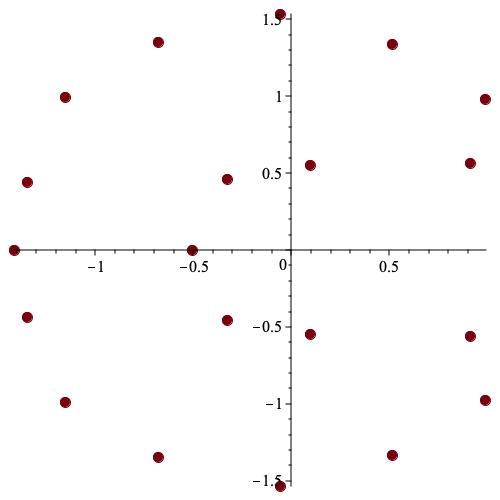}
\caption*{$G_5$}%\label{fig:1a}
\end{subfigure}\qquad\begin{subfigure}[b]{.25\linewidth}
\centering \includegraphics[scale=0.25]{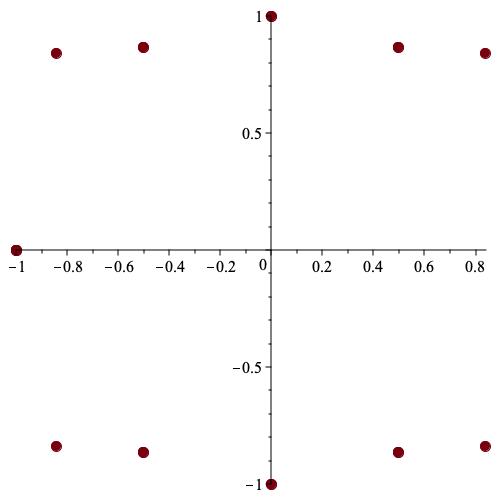}
\caption*{$G_6$}%\label{fig:1a}
\end{subfigure}

\vspace{0.7cm}

\begin{subfigure}[b]{.25\linewidth}
\centering \includegraphics[scale=0.25]{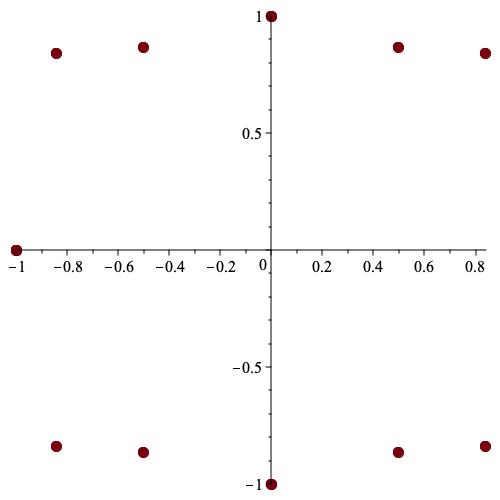}
\caption*{$G_7$}%\label{fig:1a}
\end{subfigure}\qquad\begin{subfigure}[b]{.25\linewidth}
\centering \includegraphics[scale=0.25]{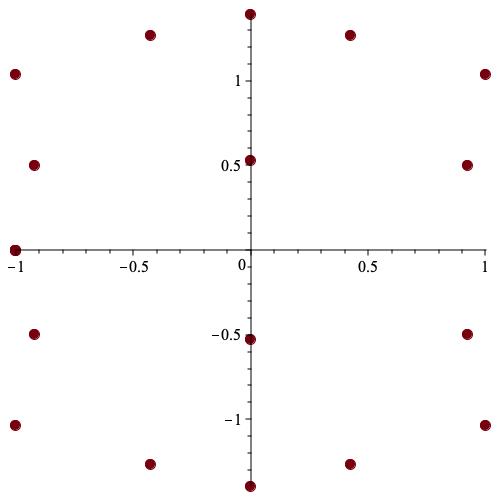}
\caption*{$G_8$}%\label{fig:1a}
\end{subfigure}\qquad\begin{subfigure}[b]{.25\linewidth}
\centering \includegraphics[scale=0.25]{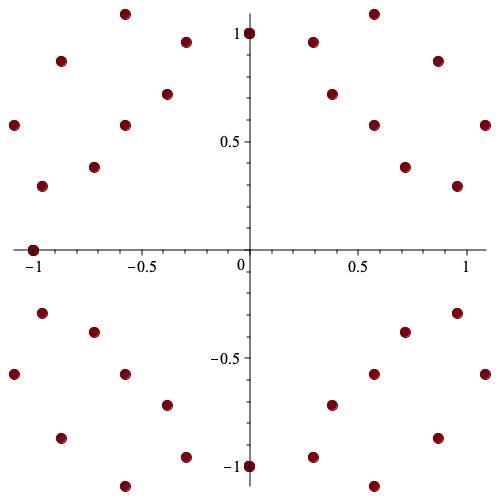}
\caption*{$G_9$}%\label{fig:1a}
\end{subfigure}

\vspace{0.7cm}

\begin{subfigure}[b]{.25\linewidth}
\centering \includegraphics[scale=0.25]{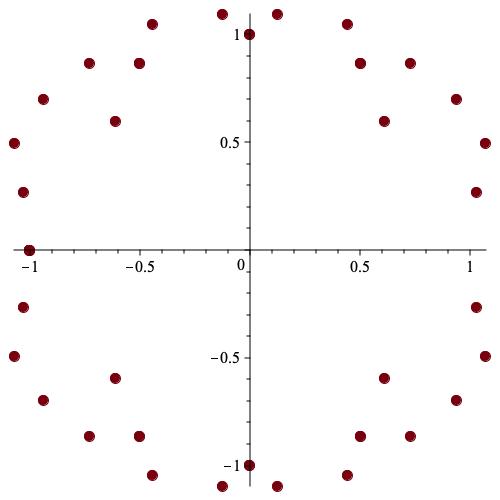}
\caption*{$G_{10}$}%\label{fig:1a}
\end{subfigure}\qquad\begin{subfigure}[b]{.25\linewidth}
\centering \includegraphics[scale=0.25]{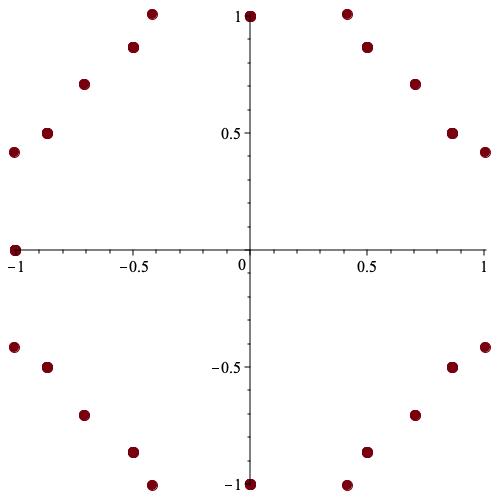}
\caption*{$G_{11}$}%\label{fig:1a}
\end{subfigure}\qquad\begin{subfigure}[b]{.25\linewidth}
\centering \includegraphics[scale=0.25]{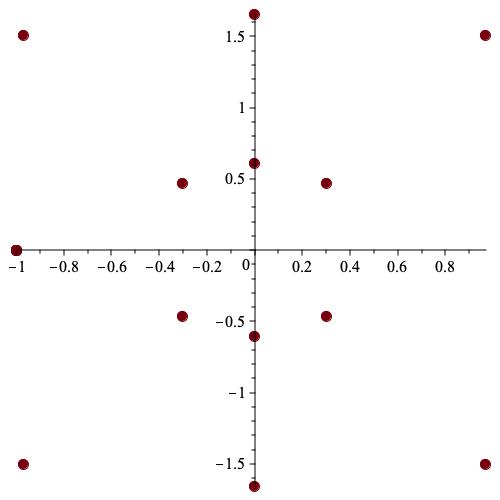}
\caption*{$G_{12}$}%\label{fig:1a}
\end{subfigure}

\vspace{0.7cm}

\begin{subfigure}[b]{.25\linewidth}
\centering \includegraphics[scale=0.25]{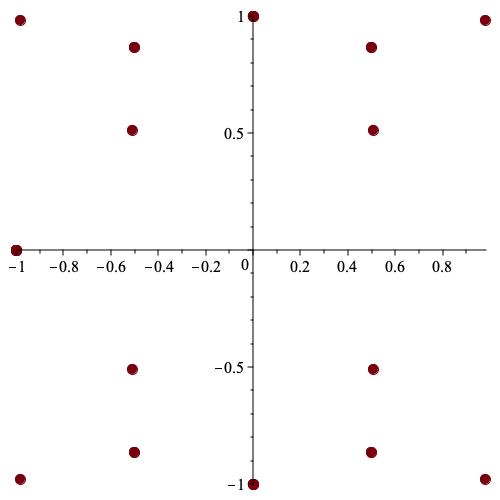}
\caption*{$G_{13}$}%\label{fig:1a}
\end{subfigure}\qquad\begin{subfigure}[b]{.25\linewidth}
\centering \includegraphics[scale=0.25]{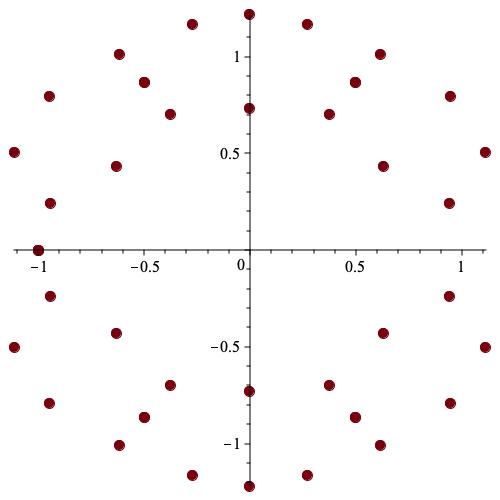}
\caption*{$G_{14}$}%\label{fig:1a}
\end{subfigure}\qquad\begin{subfigure}[b]{.25\linewidth}
\centering \includegraphics[scale=0.25]{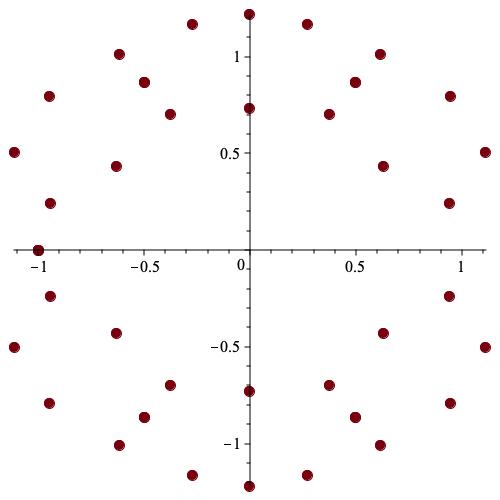}
\caption*{$G_{15}$}%\label{fig:1a}
\end{subfigure}
\end{figure}

\begin{figure}[H]
\begin{subfigure}[b]{.25\linewidth}
\centering \includegraphics[scale=0.25]{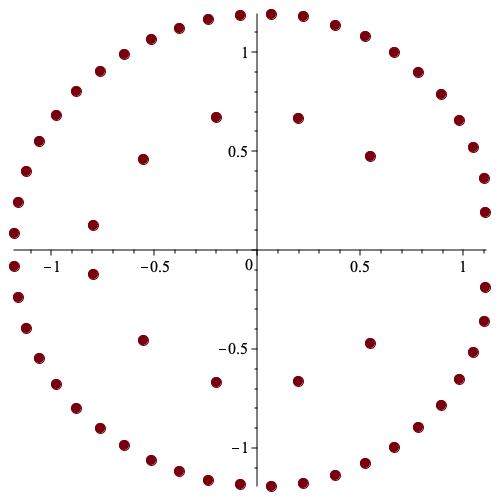}
\caption*{$G_{16}$}%\label{fig:1a}
\end{subfigure}\begin{subfigure}[b]{.25\linewidth}
\centering \includegraphics[scale=0.25]{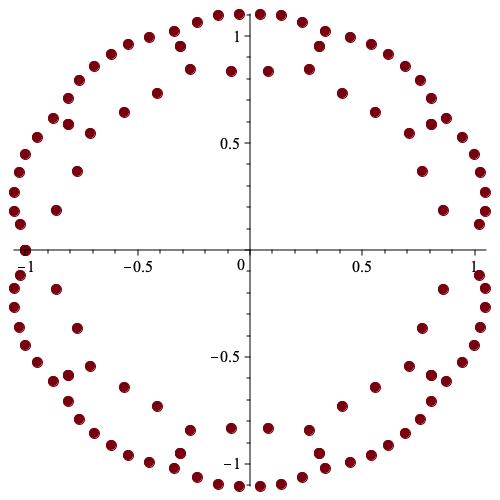}
\caption*{$G_{17}$}%\label{fig:1a}
\end{subfigure}\begin{subfigure}[b]{.25\linewidth}
\centering \includegraphics[scale=0.25]{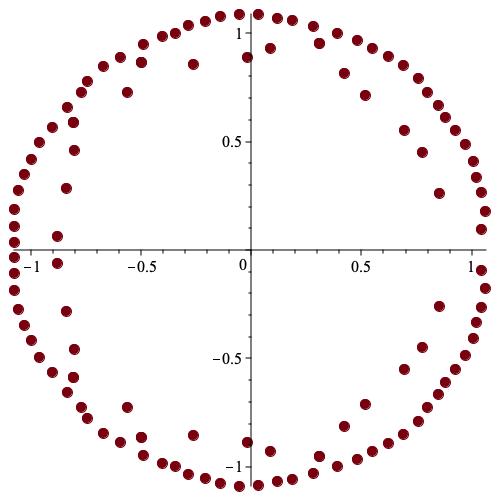}
\caption*{$G_{18}$}%\label{fig:1a}
\end{subfigure}\begin{subfigure}[b]{.25\linewidth}
\centering \includegraphics[scale=0.25]{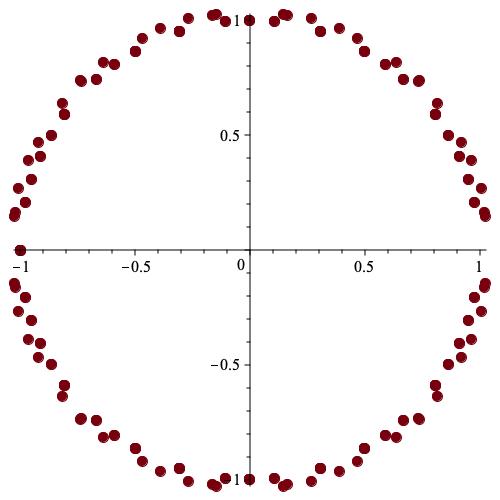}
\caption*{$G_{19}$}%\label{fig:1a}
\end{subfigure}

\begin{subfigure}[b]{.25\linewidth}
\centering \includegraphics[scale=0.25]{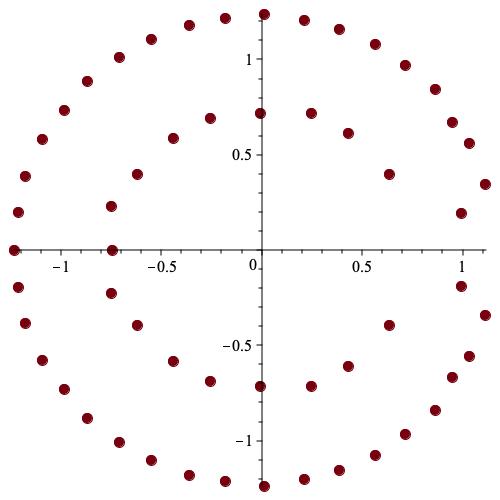}
\caption*{$G_{20}$}%\label{fig:1a}
\end{subfigure}\qquad\begin{subfigure}[b]{.25\linewidth}
\centering \includegraphics[scale=0.25]{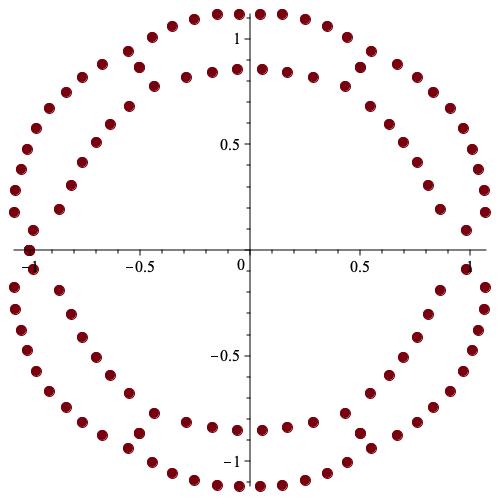}
\caption*{$G_{21}$}%\label{fig:1a}
\end{subfigure}\qquad\begin{subfigure}[b]{.25\linewidth}
\centering \includegraphics[scale=0.25]{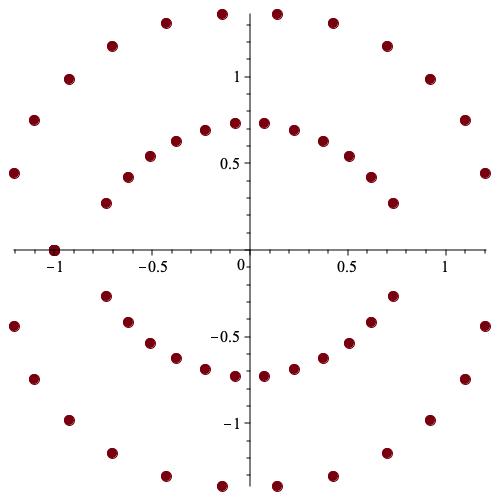}
\caption*{$G_{22}$}%\label{fig:1a}
\end{subfigure}
\end{figure}

\vspace{1cm}

\subsection*{Rank 3}\ \newline

\begin{figure}[H]
\begin{subfigure}[b]{.25\linewidth}
\centering \includegraphics[scale=0.25]{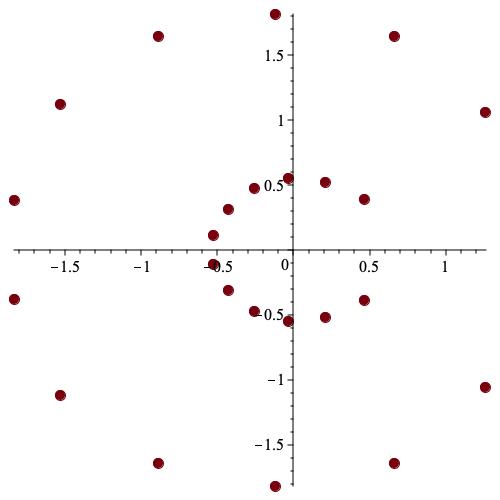}
\caption*{$G_{23}=H_3$}%\label{fig:1a}
\end{subfigure}\qquad\begin{subfigure}[b]{.25\linewidth}
\centering \includegraphics[scale=0.25]{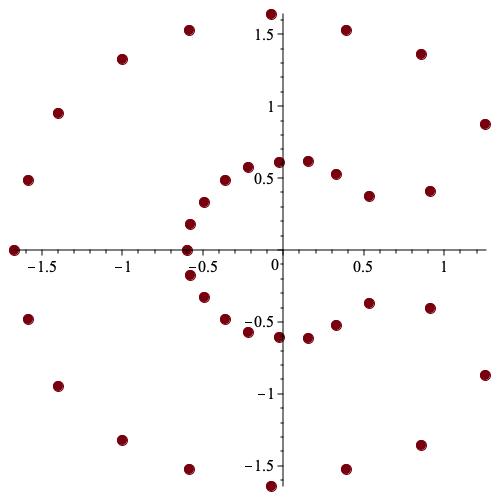}
\caption*{$G_{24}$}%\label{fig:1a}
\end{subfigure}\qquad\begin{subfigure}[b]{.25\linewidth}
\centering \includegraphics[scale=0.25]{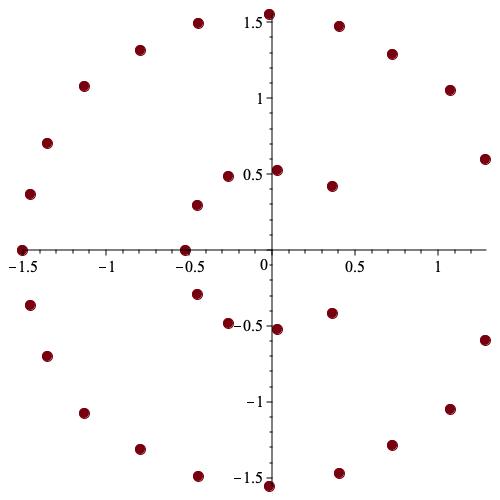}
\caption*{$G_{25}$}%\label{fig:1a}
\end{subfigure}

\vspace{1cm}

\begin{subfigure}[b]{.25\linewidth}
\centering \includegraphics[scale=0.25]{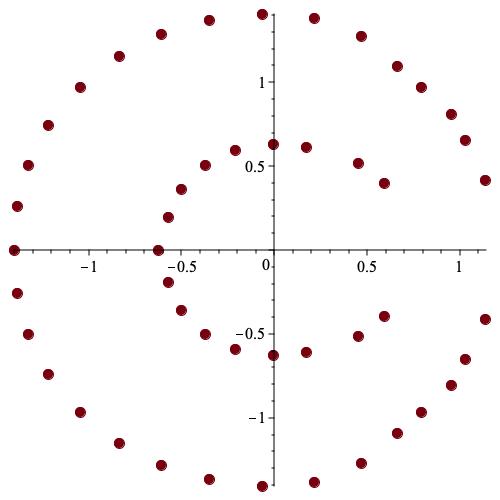}
\caption*{$G_{26}$}%\label{fig:1a}
\end{subfigure}\hspace{2cm}\begin{subfigure}[b]{.25\linewidth}
\centering \includegraphics[scale=0.25]{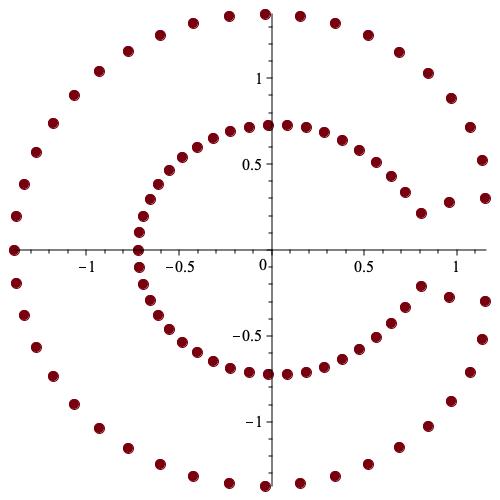}
\caption*{$G_{27}$}%\label{fig:1a}
\end{subfigure}
\end{figure}

\newpage

\subsection*{Rank 4}\ \newline

\begin{figure}[H]
\begin{subfigure}[b]{.25\linewidth}
\centering \includegraphics[scale=0.25]{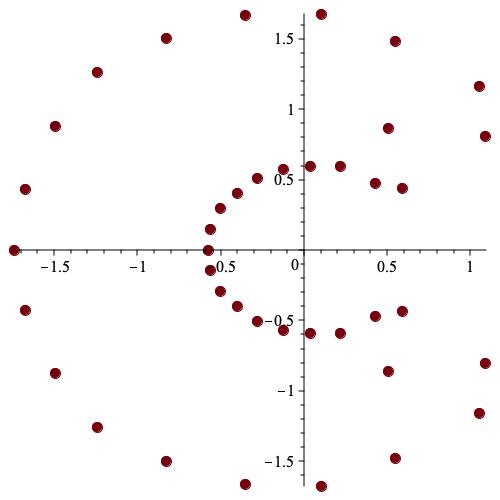}
\caption*{$G_{28}=F_4$}%\label{fig:1a}
\end{subfigure}\hspace{2cm}\begin{subfigure}[b]{.25\linewidth}
\centering \includegraphics[scale=0.25]{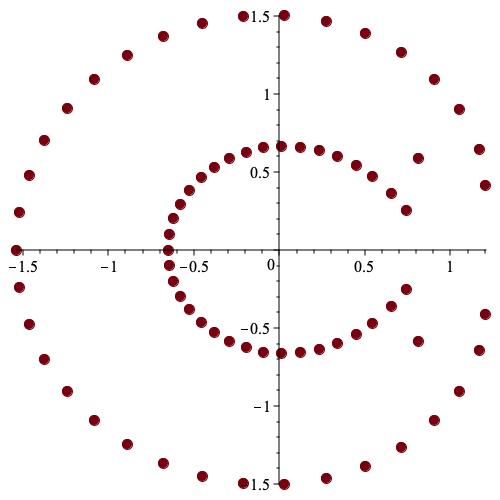}
\caption*{$G_{29}$}%\label{fig:1a}
\end{subfigure}

\vspace{1cm}
\begin{subfigure}[b]{.25\linewidth}
\centering \includegraphics[scale=0.25]{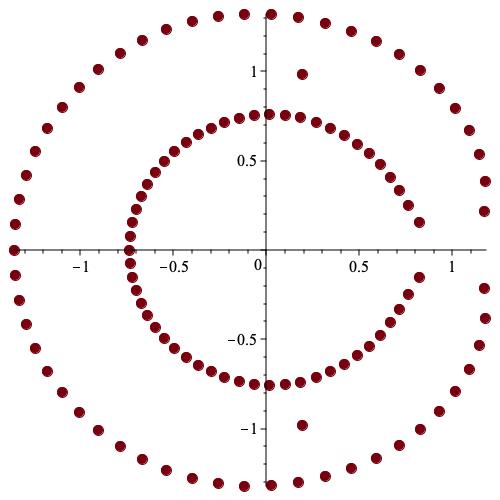}
\caption*{$G_{30}=H_4$}%\label{fig:1a}
\end{subfigure}\qquad\begin{subfigure}[b]{.25\linewidth}
\centering \includegraphics[scale=0.25]{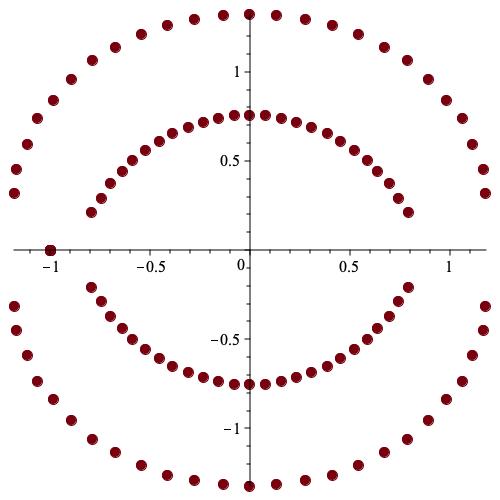}
\caption*{$G_{31}$}%\label{fig:1a}
\end{subfigure}\qquad\begin{subfigure}[b]{.25\linewidth}
\centering \includegraphics[scale=0.25]{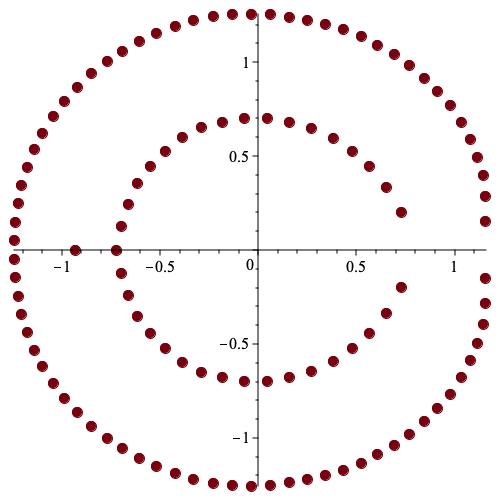}
\caption*{$G_{32}$}%\label{fig:1a}
\end{subfigure}
\end{figure}

\subsection*{Ranks 5 and 6}\ \newline

\begin{figure}[H]
\begin{subfigure}[b]{.25\linewidth}
\centering \includegraphics[scale=0.25]{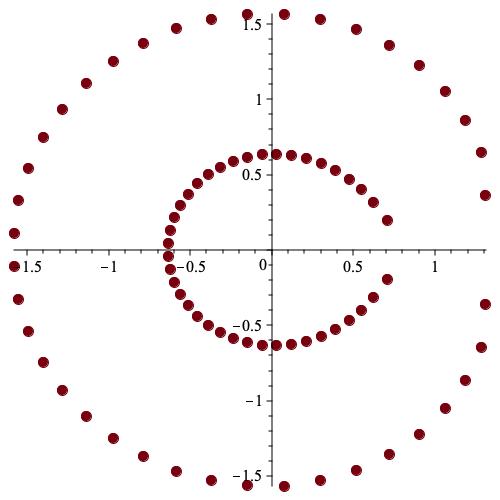}
\caption*{$G_{33}$}%\label{fig:1a}
\end{subfigure}\hspace{2cm}\begin{subfigure}[b]{.25\linewidth}
\centering \includegraphics[scale=0.25]{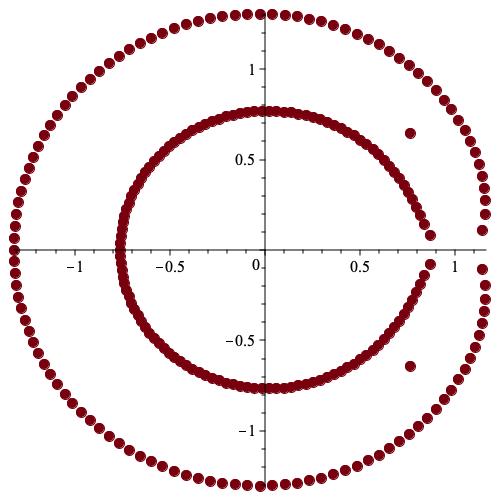}
\caption*{$G_{34}$}%\label{fig:1a}
\end{subfigure}
\end{figure}

\subsection*{E-series}\ \newline

\begin{figure}[H]
\begin{subfigure}[b]{.25\linewidth}
\centering \includegraphics[scale=0.25]{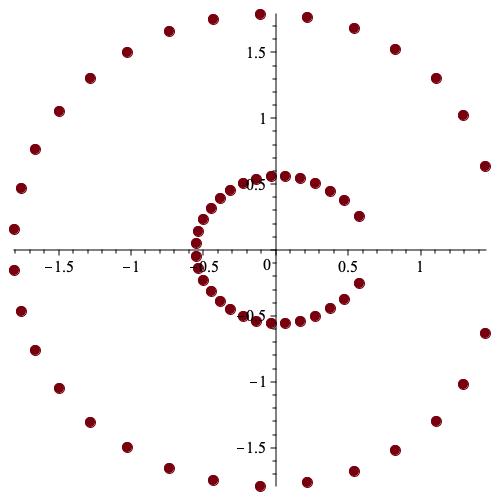}
\caption*{$G_{35}=E_{6}$}%\label{fig:1a}
\end{subfigure}
\qquad
\begin{subfigure}[b]{.25\linewidth}
\centering \includegraphics[scale=0.25]{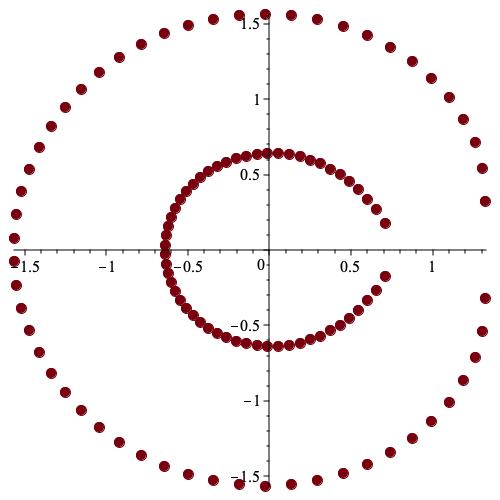}
\caption*{$G_{36}=E_{7}$}%\label{fig:1a}
\end{subfigure}
\qquad
\begin{subfigure}[b]{.25\linewidth}
\centering \includegraphics[scale=0.25]{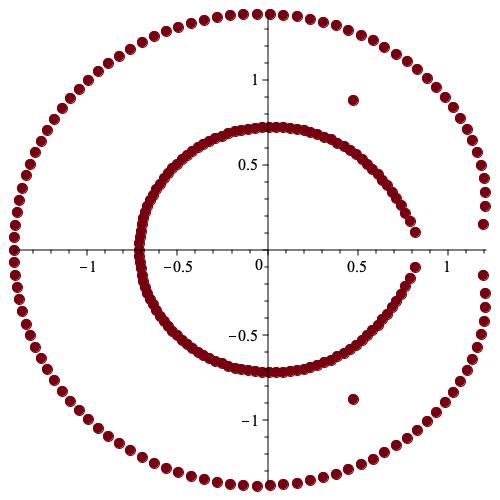}
\caption*{$G_{37}=E_{8}$}%\label{fig:1a}
\end{subfigure}
\end{figure}

\end{document}